\documentclass[reqno,a4paper,11pt]{article}
\usepackage{color}
\usepackage[latin1]{inputenc}
\usepackage[english]{babel}
\usepackage{amsmath,amssymb,amsfonts,amsthm,enumerate}
\usepackage{mathrsfs}
\usepackage[T1]{fontenc}
\usepackage{epsfig}
\usepackage{graphicx}

\definecolor{refkeybis}{gray}{.65}
\definecolor{labelkeybis}{gray}{.65}
{\makeatletter
\def\SK@refcolor{\color{refkeybis}}%
\def\SK@labelcolor{\color{labelkeybis}}}

\newtheorem{theorem}{Theorem}[section]
\newtheorem{lemma}[theorem]{Lemma}
\newtheorem{definition}[theorem]{Definition}
\newtheorem{remark}[theorem]{Remark}
\newtheorem{proposition}[theorem]{Proposition}

\newtheorem{example}[theorem]{Example}

\newcommand{\Q}{\mathbb{Q}}
\newcommand{\N}{\mathbb{N}}
\newcommand{\R}{\mathbb{R}}
\newcommand{\C}{\mathbb{C}}
\newcommand{\Z}{\mathbb{Z}}

\newcommand{\Leb}[1]{{\mathscr L}^{#1}} 

\renewcommand{\P}{\mathbb{P}}

\newcommand{\<}{\langle}
\renewcommand{\>}{\rangle}

\newcommand{\e}{\varepsilon}
\newcommand{\g}{\gamma}

\newcommand{\n}{\nabla}

\newcommand{\z}{\zeta}

\newcommand{\p}{\partial}
\newcommand{\ov}{\overline}

\newcommand{\BorelSets}[1]{\mathcal B(#1)}
\newcommand{\Probabilities}[1]{\mathscr P\bigl(#1\bigr)} 
\newcommand{\Measures}[1]{\mathscr M\bigl(#1\bigr)} 
\newcommand{\Measuresp}[1]{\mathscr M_+\bigl(#1\bigr)} 
\newcommand{\Path}[2]{\Omega_{#1}({\mathscr P}(#2))}
\newcommand{\Error}[2]{{\mathscr E}_\varepsilon'(#1,#2)}
\newcommand{\Errordelta}[2]{{\mathscr E}_\varepsilon(#1,#2)}

\newcommand{\supp}{\operatorname{supp}}

\newcommand{\E}{{\mathbb E}}

\newcommand{\res}{\mathop{\hbox{\vrule height 7pt width .5pt depth 0pt
\vrule height .5pt width 6pt depth 0pt}}\nolimits}

 \newcommand{\bb}{{\mbox{\boldmath$b$}}}
 \newcommand{\cc}{{\mbox{\boldmath$c$}}}

 \newcommand{\spp}{{\mbox{\scriptsize\boldmath$p$}}}

 \newcommand{\xx}{{\mbox{\boldmath$x$}}}
 \newcommand{\pp}{{\mbox{\boldmath$p$}}}

 \newcommand{\tauV}{{\kern-3pt\tau}}
 
 \newcommand{\sxx}{{\mbox{\scriptsize\boldmath$x$}}}
 \newcommand{\sxX}{{\mbox{\scriptsize\boldmath$X$}}}

 \newcommand{\XX}{{\mbox{\boldmath$X$}}}
 \newcommand{\YY}{{\mbox{\boldmath$Y$}}}
 
 \newcommand{\oVVVk}{\overline{\mbox{\boldmath$V$}}\kern-3pt}
 \newcommand{\tVVVk}{\tilde{\mbox{\boldmath$V$}}\kern-3pt}

 \newcommand{\mmu}{{\mbox{\boldmath$\mu$}}}
 \newcommand{\nnu}{{\mbox{\boldmath$\nu$}}}
 
 \newcommand{\smmu}{{\mbox{\scriptsize\boldmath$\mu$}}}

 \newcommand{\ssigma}{{\mbox{\boldmath$\sigma$}}}

 \newcommand{\eeta}{{\mbox{\boldmath$\eta$}}}
 \newcommand{\eps}{\varepsilon}

\newcommand{\esssup}{\mathop{\rm ess \, sup}}

\title{
 Semiclassical limit of quantum dynamics with rough potentials and
 well posedness of transport equations with measure initial data
}
\date{}
\author{Luigi Ambrosio\
   \thanks{\textsf{l.ambrosio@sns.it}}
   \and
   Alessio Figalli\
   \thanks{\textsf{figalli@math.utexas.edu}}
 \and
   Gero Friesecke\
   \thanks{\textsf{gf@ma.tum.de}}
 \and
   Johannes Giannoulis\
   \thanks{\textsf{giannoul@ma.tum.de}}
\and
   Thierry Paul\
   \thanks{\textsf{thierry.paul@math.polytechnique.fr}}
   }

\begin{document}

\maketitle

\tableofcontents

\section{Introduction}

In this paper we study the semiclassical limit of the Schr\"odinger
equation. Under mild regularity assumptions on the potential $U$
which include Born-Oppenheimer potential energy surfaces in
molecular dynamics, we establish asymptotic validity of classical
dynamics globally in space and time for ``almost all'' initial data,
with respect to an appropriate reference measure on the space of
initial data. In order to achieve this goal we study the flow in the
space of measures induced by the continuity equation: we prove
existence, uniqueness and stability properties of the flow in this
infinite-dimensional space, in the same spirit of the theory
developed in the case when the state space is Euclidean, starting
from the seminal paper \cite{lions} (see also \cite{ambrosio} and
the Lecture Notes \cite{cetraro}, \cite{bologna}).

As we said, we are concerned with the derivation of classical
mechanics from quantum mechanics, corresponding to the study of the
asymptotic behaviour of solutions $\psi^\e(t,x)=\psi^\e_t(x)$ to the
Schr\"odinger equation
\begin{equation}
\label{eq:Sch} \left\{
\begin{array}{l} i\e \p_t \psi^\e_t=-\frac{\e^2}{2}\Delta \psi^\e_t+U\psi^\e_t
=H_\e\psi^\e_t,\\\\
\psi^\e_0=\psi_{0,\e},
\end{array}
\right.
\end{equation}
as $\e\to 0$. This problem has a long history (see e.g.
\cite{Martinez}) and has been considered from a transport equation
point of view in \cite{lionspaul} and \cite{gerard} and more
recently in \cite{amfrja}, in the context of molecular dynamics. In
that context the standing assumptions on the initial conditions
$\psi_{0,\e}\in H^2(\R^n;\C)$ are:
\begin{equation}\label{initial1}
\int_{\R^n}|\psi_{0,\e}|^2\,dx=1,
\end{equation}
\begin{equation}\label{initial2}
\sup_{\e>0}\int_{\R^n}|H_\e\psi_{0,\e}|^2\,dx<\infty.
\end{equation}
The potential $U$ in \eqref{eq:Sch} is assumed to satisfy the
standard Kato conditions $U=U_{b}+U_{s}$ with
\begin{equation}\label{Kato2}
U_{s}(x)=\sum_{1\leq\alpha<\beta\leq M}
V_{\alpha\beta}(x_\alpha-x_\beta),\qquad V_{\alpha\beta}\in
L^2(\R^3)+L^\infty(\R^3)
\end{equation}
and
\begin{equation}\label{Kato1}
     U_{b}\in L^\infty(\R^n),
\end{equation}
\begin{equation}\label{Kato3}
     \nabla U_{b}\in L^\infty(\R^n;\R^n).
\end{equation}
Here $n=3M$, $x=(x_1,\ldots,x_M)\in (\R^3)^M$ represent the
positions of atomic nuclei. Under assumptions \eqref{Kato2},
\eqref{Kato1} the operator $H_\e$ is selfadjoint on $L^2(\R^n;\C)$
with domain $H^2(\R^n;\C)$ and generates a unitary group in
$L^2(\R^n;\C)$; hence $\int_{\R^n}|\psi^\e_t|^2\,dx=1$ for all
$t\in\R$, $t\mapsto\psi^\e_t$ is continuous with values in
$H^2(\R^n;\C)$ and continuously differentiable with values in
$L^2(\R^n;\C)$. Prototypically, $U$ is the Born-Oppenheimer ground
state potential energy surface of the molecule, that is to say
$V_{\alpha\beta}(x_\alpha-x_\beta)=Z_\alpha
Z_\beta|x_\alpha-x_\beta|^{-1}$, $Z_\alpha,\,Z_\beta\in\N$,
$U_b(x)=\inf \mbox{spec}\, H_{e\ell}(x)$, where
$$
   H_{e\ell}(x)=\sum_{i=1}^N(-\frac12\Delta_{r_i}-\sum_{\alpha=1}^MZ_\alpha
   |r_i-x_\alpha|^{-1})+ \sum_{1\le i<j\le N} |r_i-r_j|^{-1}
$$
is the electronic Hamiltonian acting on the antisymmetric subspace
of $L^2((\R^3\times\Z_2)^N;\C)$ and the $r_i\in\R^3$ are electronic
position coordinates. For neutral or positively charged molecules
($N\le \sum_{\alpha=1}^M Z_\alpha$), Zhislin's theorem (see
\cite{Fr03} for a short proof) states that for all $x$, $U_b(x)$ is
an isolated eigenvalue of finite multiplicity of $H_{e\ell}(x)$.

In the study of this semiclassical limit difficulties arise on the
one hand from the fact that $\nabla U$ is unbounded (because of
Coulomb singularities) and on the other hand from the fact that
$\nabla U$ might be discontinuous even out of Coulomb singularities
(because of possible eigenvalue crossings of the electronic
Hamiltonian $H_{e\ell}$). Fortunately, it turns out that these two
difficulties can be dealt with separately.

If we denote by $\bb:\R^{2n}\to\R^{2n}$ the autonomous
divergence-free vector field $\bb(x,p):=\bigl(p,-\nabla U(x)\bigr)$,
the Liouville equation describing classical dynamics is
\begin{equation}\label{contieqliou}
\partial_t \mu_t+p\cdot\nabla_x \mu_t-\nabla U(x)\cdot\nabla_p \mu_t=0.
\end{equation}
If we denote by $W_\e:L^2(\R^n;\C)\to L^\infty(\R^n_x\times\R^n_p)$
the Wigner transform, namely
\begin{equation}\label{Wigner}
W_\e\psi(x,p):=\frac{1}{(2\pi)^n}\int_{\R^n}\psi(x+\frac{\e}{2}y)
\overline{\psi(x-\frac{\e}{2}y)}e^{-ipy}dy,
\end{equation}
a calculation going back to Wigner himself (see for instance
\cite{lionspaul} or \cite{amfrja} for a detailed derivation) shows
that $W_\e{\psi_t^\e}$ solves in the sense of distributions the
equation
\begin{equation}\label{Wigner_calc_in}
\partial_t W_\e\psi^\e_t+p\cdot\nabla_x W_\e\psi^\e_t
=\Errordelta{U}{\psi^\eps_t},
\end{equation}
where $\Errordelta{U}{\psi}(x,p)$ is given by
\begin{equation}\label{error_delta}
\Errordelta{U}{\psi}(x,p):= -\frac{i}{(2\pi)^n}\int_{\R^n} \biggl[
\frac{U(x+\tfrac{\e}{2}y)-U(x-\tfrac{\e}{2}y)}{\e}\biggr]
\psi(x+\frac{\e}{2}y) \overline{\psi(x-\frac{\e}{2}y)}e^{-ipy}dy.
\end{equation}
Adding and subtracting $\nabla U(x)\cdot y$ in the term in square
brackets and using $ye^{-ip\cdot y}=i\nabla_p e^{-ip\cdot y}$, an
integration by parts gives $\Errordelta{U}{\psi}=\nabla U(x)\cdot
\nabla_pW_\e\psi+\Error{U}{\psi}$, where $\Error{U}{\psi}(x,p)$ is
given by
\begin{equation}\label{error}
\Error{U}{\psi}(x,p):= -\frac{i}{(2\pi)^n}\int_{\R^n} \biggl[
\frac{U(x+\tfrac{\e}{2}y)-U(x-\tfrac{\e}{2}y)}{\e}-\langle\nabla
U(x),y\rangle \biggr] \psi(x+\frac{\e}{2}y)
\overline{\psi(x-\frac{\e}{2}y)}e^{-ipy}dy.
\end{equation}
Hence, $W_\e\psi^\e_t$ solves \eqref{contieqliou} with an error
term:
\begin{equation}\label{Wigner_calc}
\partial_t W_\e\psi^\e_t+\nabla_{x,p}\cdot \bigl(\bb
W_\e\psi^\e_t\bigr)=\Error{U}{\psi^\eps_t}.
\end{equation}
Heuristically, since the term in square brackets in \eqref{error}
tends to $0$ when $U$ is differentiable, this suggests that the
limit of $W_\e{\psi^\e_t}$ should satisfy \eqref{contieqliou}, and a
first rigorous proof of this fact was given in \cite{lionspaul} and
\cite{gerard} (see also \cite{gerard3}): basically, ignoring other global conditions on $U$,
these results state that:

\noindent (a) $C^1$ regularity of $U$ ensures that limit points of
$W_\e\psi^\e_t$ as $\e\downarrow 0$ exist and satisfy
\eqref{contieqliou};

\noindent (b) $C^2$ regularity of $U$ ensures uniqueness of the
limit, i.e. full convergence as $\e\to 0$.

\noindent In (a), convergence of the Wigner transforms is understood
in a natural dual space ${\cal A}'$ (see \eqref{defA} for the
definition of ${\cal A}$).

In \cite{amfrja} we were able to achieve the existence of limit
points even when Coulomb singularities and crossings are present,
namely assuming only that $U_b$ satisfies (\ref{Kato1}),
(\ref{Kato3}), and (a) when Coulomb singularities but no crossings
are present, namely assuming that $U_{b}\in C^1$. If one wishes to
improve (a) and (b), trying to prove a full convergence result as
$\e\downarrow 0$ under weaker regularity assumptions on $\bb$ (say
$\nabla U\in W^{1,p}$ or $\nabla U\in BV$ out of Coulomb
singularities), one faces the difficulty that the continuity
equation \eqref{contieqliou} is well posed only in good functional
spaces like $L^\infty_+\bigl([0,T];L^1\cap L^\infty(\R^d)\bigr)$
(see \cite{lions}, \cite{ambrosio}, \cite{bouchut}). On the other
hand, in the study of semiclassical limits it is natural to consider
families of wavefunctions $\psi_{0,\e}$ in \eqref{eq:Sch} whose
Wigner transforms do concentrate as $\e\downarrow 0$, for instance
the semiclassical wave packets
\begin{equation}\label{alphaless1}
\psi_{0,\e}(x)=\e^{-n\alpha/2}\phi_0\Bigl(\frac{x-x_0}{\e^\alpha}\Bigr)e^{i(x\cdot
p_0)/\e}\qquad \phi_0\in C^2_c(\R^n),\,\,0<\alpha<1
\end{equation}
which satisfy $\lim_\e W_\e\psi_{0,\e} =\Vert\phi_0\Vert_{L^2}^2
\delta_{(x_0,p_0)}$. Here the limiting case $\alpha=1$ corresponds
to concentration in position only, $\lim_\e
W_\e\psi_{0,\e}=\delta_{x_0}\times(2\pi)^{-n} |{\cal
F}\phi_0|^2(\cdot-p_0)\Leb{n}$, and the case $\alpha=0$ yields
concentration in momentum only, $\lim_\e
W_\e\psi_{0,\e}=|\phi_0(\cdot-x_0)|^2\times\delta_{p_0}$. Here and
below, $({\cal F}\phi_0)(p)=\int_{\R^n}e^{-ip\cdot x}\phi_0(x)\, dx$
denotes the (standard, not scaled) Fourier transform. But, even in
these cases there is a considerable difficulty in the analysis of
\eqref{error_delta}, since the difference quotients of $U$ have a
limit only at $\Leb{n}$-a.e. point.

For these initial conditions there is presumably no hope to achieve
full convergence as $\e\to 0$ for \emph{all} $(x_0,p_0)$, since the
limit problem is not well posed. However, in the same spirit of the
theory of flows that we shall illustrate in the second part of the
introduction, one may look at the family of solutions, indexed in
the case of the initial conditions \eqref{alphaless1} by
$(x_0,p_0)$, as a whole. More generally, we are considering a family
of solutions $\psi^\e_{t,w}$ to \eqref{eq:Sch} indexed by a
``random'' parameter $w\in W$ running in a probability space
$(W,{\mathcal F},\P)$, and achieve convergence ``with probability
one'', using the theory developed in the first part of the paper,
under the no-concentration in mean assumptions
\begin{equation}\label{nomeancon}
\sup_{\e>0}\sup_{t\in\R}\biggl\Vert\int_W W_\e\psi^\e_{t,w}\ast
G^{(2n)}_\e\,d\P(w)\biggr\Vert_{L^\infty(\R^{2n})}<\infty,
\end{equation}
\begin{equation}\label{nomeanconbis}
\sup_{\e>0}\sup_{t\in\R}\biggl\Vert\int_W |\psi^\e_{t,w}\ast
G^{(2n)}_{\lambda\e^2}|^2\,d\P(w)\biggr\Vert_{L^\infty(\R^{n})}\leq
C(\lambda)<\infty\qquad\forall\lambda>0.
\end{equation}
Here $G^{(2n)}_\e$ is the Gaussian kernel in $\R^{2n}$ with variance
$\e/2$. Under these assumptions and those on $U$ given in
Section~7.2, our full convergence result reads as follows:
\begin{equation}\label{limeps}
\lim_{\e\downarrow 0} \int_W\sup_{t\in [-T,T]} d_{{\cal
A}'}\bigl(W_\e\psi^\e_{t,w},\mmu(t,\mu_w)\bigr)\,d\P(w)=0
\qquad\forall\, T>0
\end{equation}
(here $d_{{\cal A}'}$ is any bounded distance inducing the weak$^*$
topology in the unit ball of ${\cal A'}$) where $\mmu(t,\mu_w)$ is
the flow in the space of probability measures at time $t$ starting
from $\mu_w$, and $\mu_w=\lim_\e W_\e\psi^\e_{0,w}$ depends only on
the initial conditions. For instance, in the case of the initial
conditions \eqref{alphaless1} with $\|\phi_0\|_2=1$, indexed by
$w=(x_0,p_0)$, $\mu_w=\delta_w$ and
$\mmu(t,\mu_w)=\delta_{\sxX(t,w)}$, where $\XX(t,w)$ is the flow in
$\R^{2n}$ induced by $(p,-\nabla U)$. So, we may say that the flow
of Wigner measures, thought of as elements of ${\cal A}'$, induced
by the Schr\"odinger equation converges as $\e\to 0$ to the flow in
$\Probabilities{\R^{2n}}\subset {\cal A}'$ induced by the Liouville
equation, provided the initial conditions ensure \eqref{nomeancon}
and \eqref{nomeanconbis}.

Of course one can question about the conditions \eqref{nomeancon}
and \eqref{nomeanconbis}; we show that both are implied by the
uniform operator inequality (here $\rho^\psi$ is the orthogonal
projection on $\psi$)
$$
\frac{1}{\e^n}\int_W \rho^{\psi^\e_{t,w}}\,d\P(w)\leq C{\rm Id}
\qquad\text{with $C$ independent of $t,\,\e$.}
$$
In turn, this latter property is propagated in time (i.e. if the
inequality holds at $t=0$ it holds for all times), and it has a
natural quantum mechanical interpretation. In addition, the uniform
operator inequality is fulfilled by the classical family of initial
data \eqref{alphaless1} when $\P$ is a bounded probability density
on $\R^{2n}$. These results indicate also that the no-concentration
in mean conditions are not only technically convenient, but somehow
natural.

An alternative approach to the flow viewpoint advocated here for
validating classical dynamics (\ref{contieqliou}) from quantum
dynamics (\ref{eq:Sch}) would be to work with deterministic initial
data, but restrict them to those giving rise to suitable bounds, in
mean, on the projection operators $\rho^{\psi_{0,\e}}$. The problem
of finding sufficient conditions to ensure these uniform bounds is
studied in \cite{figpaul}. Another related research direction is a
finer analysis of the behaviour of solutions, in the spirit of
\cite{Fermanian1}, \cite{Fermanian2}. However, this analyis is
presently possible only for very particular cases of eigenvalue crossings.

It is likely that our results can be applied to many more families
of initial conditions, but this is not the goal of this paper. The
proof of \eqref{limeps} relies on several apriori and fine estimates
and on the theoretical tools described in the second part of the
introduction and announced in \cite{amfiga2}. In particular we apply
the stability properties of the $\nnu$-RLF in
$\Probabilities{\R^{2n}}$, see Theorem~\ref{tstable}, to the Husimi
transforms of $\psi^\e_{t,w}$, namely $W_\e\psi^\e_{t,w}\ast
G^{(2n)}_\e$. Indeed, $w^*$-convergence in ${\cal A}'$ of the Wigner
transforms is equivalent, under extra tightness assumptions, to weak
convergence in $\Probabilities{\R^{2n}}$ of the Husimi transforms.

We leave aside further extensions analogous to those considered in
\cite{lionspaul}, namely the convergence of density matrices
$\rho^\e$, whose dynamics is described by $i
\e\partial_t\rho^\e=[H_\e,\rho^\e]$, and the nonlinear case when
$U=U_0\ast\mu$, $\mu$ being the position density of $\psi$ (i.e.
$|\psi|^2$). In connection with the first extension, notice that the
action of the Wigner/Husimi transforms becomes linear, when seen at
this level.

Let us now describe the ``flow'' viewpoint first in
finite-dimensional spaces, where by now the theory is well
understood. Denoting by $\bb_t:\R^d\to\R^d$, $t\in [0,T]$, the
possibly time-dependent velocity field, the first basic idea is not
to look for pointwise uniqueness statements, but rather to the
family of solutions to the ODE as a whole. This leads to the concept
of flow map $\XX(t,x)$ associated to $\bb$, i.e. a map satisfying
$\XX(0,x)=x$ and $\XX(t,x)=\g(t)$, where $\g(0)=x$ and
\begin{equation}\label{ODE}
\dot \g(t)=\bb_t(\g(t)) \qquad\text{for $\Leb{1}$-a.e. $t\in (0,T)$.}
\end{equation}
for $\Leb{d}$-a.e. $x\in\R^d$. It is easily seen that this is not an
invariant concept, under modification of $\bb$ in negligible sets,
while many applications of the theory to fluid dynamics (see for
instance \cite{lions2}, \cite{lions3}) and conservation laws need
this invariance property. This leads to the concept of \emph{regular
Lagrangian flow} (RLF in short): one may ask that, for all $t\in
[0,T]$, the image of the Lebesgue measure $\Leb{d}$ under the flow
map $x\mapsto\XX(t,x)$ is still controlled by $\Leb{d}$ (see
Definition~\ref{RLflow}). It is not hard to show that, because of
the additional regularity condition imposed on $\XX$, this concept
is indeed invariant under modifications of $\bb$ in Lebesgue
negligible sets (see Remark~\ref{invariaflow}). Hence RLF's are
appropriate to deal with vector fields belonging to Lebesgue $L^p$
spaces. On the other hand, since this regularity condition involves
all trajectories $\XX(\cdot,x)$ up to $\Leb{d}$-negligible sets of
initial data, the best we can hope for using this concept is
existence and uniqueness of $\XX(\cdot,x)$ up to
$\Leb{d}$-negligible sets. Intuitively, this can be viewed as
existence and uniqueness ``with probability one'' with respect to a
reference measure on the space of initial data. Notice that already
in the finite-dimensional theory different reference measures (e.g.
Gaussian, see \cite{amfiga}) could be considered as well.

To establish such existence and uniqueness, one uses that the concept of flow is directly linked,
via the theory of characteristics, to the transport equation
\begin{equation}\label{transport}
\frac{d}{ds} f(s,x)+\langle\bb_s(x),\nabla_x f(s,x)\rangle=0
\end{equation}
and to the continuity equation
\begin{equation}\label{contieq}
\frac{d}{dt}\mu_t+\nabla\cdot (\bb_t\mu_t)=0.
\end{equation}
The first equation has been exploited in \cite{lions} to transfer
well-posedness results from the transport equation to the ODE,
getting uniqueness of RLF (with respect to Lebesgue measure) in
$\R^d$. This is possible because the flow maps $(s,x)\mapsto
\XX(t,s,x)$ (here we made also explicit the dependence on the
initial time $s$, previously set to 0) solve \eqref{transport} for
all $t\in [0,T]$. In the present article, in analogy with the
approach initiated in \cite{ambrosio} (see also \cite{figalli} for a
stochastic counterpart of it, where \eqref{contieq} becomes the
forward Kolmogorov equation), we prefer rather to deal with the
continuity equation, which seems to be more natural in a
probabilistic framework. The link between the ODE \eqref{ODE} and
the  continuity equation \eqref{contieq} can be made precise as
follows: any positive finite measure $\eeta$  on initial values and
paths, $\eeta\in \Probabilities{\R^d\times
C\bigl([0,T];\R^d\bigr)}$, concentrated on solutions $(x,\g)$ to the
ODE with initial condition $x=\g(0)$, gives rise to a
(distributional) solution to \eqref{contieq}, with $\mu_t$ given by
the marginals of $\eeta$ at time $t$: indeed, \eqref{contieq}
describes the evolution of a probability density under the action of
the ``velocity field'' $\bb$. We shall call these measures $\eeta$
\emph{generalized} flows, see Definition \ref{GRLflow}. These facts
lead to the existence, the uniqueness (up to $\Leb{d}$-negligible
sets) and the stability of the RLF $\XX(t,x)$ in $\R^d$
\emph{provided} \eqref{contieq} is well-posed in
$L^\infty_+\bigl([0,T];L^1\cap L^\infty(\R^d)\bigr)$. Roughly
speaking, this should be thought of as a regularity assumption on
$\bb$. See Remark~\ref{rgoodb} and Section~\ref{sgoodliou} for
explicit conditions on $\bb$ ensuring well-posedness.

We shall extend all these results to flows on
$\Probabilities{\R^d}$, the space of probability measures on $\R^d$.
The heuristic idea is that \eqref{contieq} can be viewed as a
(constant coefficients) ODE in the infinite-dimensional space
$\Probabilities{\R^d}$, and that we can achieve uniqueness results
for \eqref{contieq} for ``almost every'' measure initial condition.
We need, however, a suitable reference measure on
$\Probabilities{\R^d}$, that we shall denote by $\nnu$. Our theory
works for many choices of $\nnu$ (in agreement with the fact that no
canonical choice of $\nnu$ seems to exist), provided $\nnu$
satisfies the \emph{regularity} condition
$$
\int_{{\mathscr P}(\R^d)}\mu\,d\nnu(\mu)\leq C\Leb{d},
$$
see Definition~\ref{RegMis}. (See also Example~\ref{eRegMis} for
some natural examples of regular measures $\nnu$.) Given $\nnu$ as
reference measure, and assuming that \eqref{contieq} is well-posed
in $L^\infty_+\bigl([0,T];L^1\cap L^\infty(\R^d)\bigr)$, we prove
existence, uniqueness (up to $\nnu$-negligible sets) and stability
of the regular Lagrangian flow of measures $\mmu$. Since this
assumption is precisely the one needed to have existence and
uniqueness of the RLF $\XX(t,x)$ in $\R^d$, it turns out that the
RLF $\mmu(t,\mu)$ in $\Probabilities{\R^d}$ is given by
\begin{equation}\label{ovvia}
\mmu(t,\mu)=\int_{\R^d} \delta_{\sxX(t,x)}\,d\mu(x)\qquad \forall\,
t\in [0,T],\,\,\mu\in\Probabilities{\R^d},
\end{equation}
which makes the existence part of our results rather easy whenever
an underlying flow $\XX$ in $\R^d$ exists. On the other hand, even in
this situation, it turns out that uniqueness and stability results are
much stronger when stated at the $\Probabilities{\R^d}$ level.

In our proofs, which follow by an infinite-dimensional adaptation of
\cite{ambrosio}, \cite{cetraro}, we use also the concept of
generalized flow in $\Probabilities{\R^d}$, i.e. measures
 $\eeta$ on $\Probabilities{\R^d}\times
C\bigl([0,T];\Probabilities{\R^d}\bigr)$
concentrated on initial data/solution pairs $(\mu,\omega)$
to \eqref{contieq} with $\omega(0)=\mu$, see Definition \ref{GRLflowmis}.

\smallskip
\noindent {\bf Acknowledgement.} We thank Dr. Marilena Ligab\'o for
pointing out a serious gap in a preliminary version of
Theorem~\ref{extracanc}. The first author was partially supported by
ERC ADG Grant GeMeThNES and the second author was partially
supported by NSF Grant DMS-0969962.

\section{Notation and preliminary results}\label{notation}

Let $X$ be a Polish space (i.e. a separable topological space whose
topology is induced by a complete distance). We shall denote by
$\BorelSets{X}$ the $\sigma$-algebra of Borel sets of $X$, by
$\Probabilities{X}$ (resp. $\Measures{X},\,\Measuresp{X}$) the space
of Borel probability (resp. finite Borel, finite Borel nonnegative)
measures on $X$. For $A\in\BorelSets{X}$ and $\nu\in\Measures{X}$,
we denote by $\nu\res A\in\Measures{X}$ the restricted measure,
namely $\nu\res A(B)=\nu(A\cap B)$. Given $f:X\to Y$ Borel and
$\mu\in\Measures{X}$, we denote by $f_\sharp\mu\in\Measures{Y}$ the
push-forward measure on $Y$, i.e. $f_\sharp\mu(A)=\mu(f^{-1}(A))$
(if $\mu$ is a probability measure, $f_\sharp\mu$ is the law of $f$
under $\mu$) and we recall the basic integration rule
$$
\int_Y \phi\,d f_\sharp\mu=\int_X \phi\circ
f\,d\mu\qquad\text{$\phi$ bounded and Borel.}
$$
We denote by $\chi_A$ the characteristic function of a set $A$,
equal to 1 on $A$, and equal to 0 on its complement. Balls in
Euclidean spaces will be denoted by $B_R(x_0)$, and by $B_R$ if
$x_0=0$.

We shall endow $\Probabilities{X}$ with the metrizable topology
induced by the duality with $C_b(X)$, the space of continuous
bounded functions on $X$: this makes $\Probabilities{X}$ itself a
Polish space (see for instance \cite[Remark~5.1.1]{amgisa}), and we
shall also consider measures $\nnu\in\Measuresp{\Probabilities{X}}$.

Typically we shall use greek letters to denote measures, boldface
greek letters to denote measures on the space of measures, and we
occasionally use $d_{{\mathscr P}}$ for a bounded distance in
$\Probabilities{X}$ inducing the weak topology induced by the
duality with $C_b(X)$ (no specific choice of $d_{{\mathscr P}}$ will
be relevant for us). We recall that weak convergence of $\mu_n$ to
$\mu$ implies
\begin{equation}\label{finerco}
\lim_{n\to\infty}\int_X f\,d\mu_n=\int_X f\,d\mu\qquad\text{for all
$f$ bounded Borel, with a $\mu$-negligible discontinuity set.}
\end{equation}
Also, in the case $X=\R^d$, recall that a sequence
$(\mu_n)\subset\Probabilities{\R^d}$ weakly converges to a
probability measure $\mu$ in the duality with $C_b(\R^d)$ if and
only if it converges in the duality with (a dense subspace of)
$C_c(\R^d)$.

We shall consider the space
$C\bigl([0,T];\Probabilities{\R^d}\bigr)$, whose generic element
will be denoted by $\omega$, endowed with the sup norm; for this
space we use the compact notation $\Path{T}{\R^d}$. We also use
$e_t$ as a notation for the evaluation map at time $t$, so that
$e_t(\omega)=\omega(t)$. Again, we shall consider measures
$\eeta\in\Measuresp{\Path{T}{\R^d}}$ and the basic criterion we
shall use is the following:

\begin{proposition}[Tightness]\label{ptightness}
Let $(\eeta_n)\subset\Measuresp{\Path{T}{\R^d}}$ be a bounded family
satisfying:
\begin{itemize}
\item[(i)] (space tightness)  for all $\eps>0$,
$\sup_n\eeta_n\Bigl(\bigl\{\omega:\ \sup\limits_{t\in
[0,T]}\omega(t)(\R^d\setminus B_R)>\eps\bigr\}\Bigr)\to 0$ as
$R\to\infty$; \item[(ii)] (time tightness) for all $\phi\in
C^\infty_c(\R^d)$, $n\geq 1$, the map
$t\mapsto\int_{\R^d}\phi\,d\omega(t)$ is absolutely continuous in
$[0,T]$ for $\eeta_n$-a.e. $\omega$ and
$$
\lim_{M\uparrow\infty}\sup_n\eeta_n\biggl(\Bigr\{\omega:\
\int_0^T\biggl|\biggl(\int_{\R^d}\phi\,d\omega(t)\biggr)'\biggr|\,dt>M\Bigr\}\biggr)=0.
$$
\end{itemize}
Then $(\eeta_n)$ is tight.
\end{proposition}
\begin{proof}
For all $\phi\in C^\infty_c(\R^d)$ we shall denote by
$I_\phi:\Path{T}{\R^d}\to C\bigl([0,T]\bigr)$ the time-dependent
integral w.r.t. $\phi$. Since the sets
$$
\left\{f\in W^{1,1}(0,T):\ \sup|f|\leq
C,\,\,\,\int_0^T|f'(t)|\,dt\leq M\right\}
$$
are compact in $C\bigl([0,T]\bigr)$, by assumption (ii) the sequence
$((I_\phi)_\sharp\eeta_n)$ is tight in
$\Measuresp{C\bigl([0,T]\bigr)}$ for all $\phi\in C^\infty_c(\R^d)$.
Hence, if we fix a countable dense set $(\phi_k)\subset
C^\infty_c(\R^d)$ and $\eps>0$, we can find for $k\geq 1$ compact
sets $K_k^\eps\subset C\bigl([0,T]\bigr)$ such that
$\sup_n\eeta_n\bigl(\Path{T}{\R^d}\setminus
I_{\phi_k}^{-1}(K^\eps_k)\bigr)<\eps 2^{-k}$. Thus, if $K^\eps$
denotes the intersection of all sets $I_{\phi_k}^{-1}(K_k^\eps)$, we
get
$$
\sup_n\eeta_n\bigl(\Path{T}{\R^d}\setminus K^\eps)<\eps.
$$
Analogously, using assumption (i) we can build another compact set
$L^\eps\subset\Path{T}{\R^d}$ such that
$\sup_n\eeta_n\bigl(\Path{T}{\R^d}\setminus L^\eps)<\eps$ and, for
all integers $k\geq 1$, there exists $R=R_k$ such that
$\omega(t)(\R^d\setminus B_R)<1/k$ for all $\omega\in L^\eps$ and
$t\in [0,T]$.

In order to conclude, it suffices to show that $K^\eps\cap L^\eps$
is compact in $\Path{T}{\R^d}$: if $(\omega_p)\subset K^\eps\cap
L^\eps$ we can use the inclusion in $I_{\phi_k}^{-1}(K_k^\eps)$
and a diagonal argument to extract a subsequence
$(\omega_{p(\ell)})$ such that $\int\phi_k\,d\omega_{p(\ell)}(t)$
has a limit for all $t\in [0,T]$ and all $k\geq 1$ and the limit
is continuous in time. By the space tightness given by the
inclusion $(\omega_p)\subset L^\eps$, $\omega_{p(\ell)}(t)$
converges to $\omega(t)$ in $\Probabilities{\R^d}$ for all $t\in
[0,T]$, and $t\mapsto\omega(t)$ is continuous.
\end{proof}

The next lemma is a refinement of \cite[Lemma~22]{cetraro} and
\cite[Corollary~5.23]{villani2}, and allows to obtain convergence in probability
from weak convergence of the measures induced on the graphs.

\begin{lemma}\label{lplans}
Let $f_n:X\to Y$, $f:X\to Y$ be Borel maps,
$\nu_n,\,\nu\in\Probabilities{X}$ and assume that $(Id\times
f_n)_\sharp\nu_n$ weakly converge to $(Id\times f)_\sharp\nu$ in
$X\times Y$. Assume in addition that we have the Skorokhod
representations $\nu_n=(i_n)_\sharp \P$, $\nu=i_\sharp \P$, with
$(W,{\cal F},\P)$  probability measure space, $i_n,\,i:W\to X$
measurable, and $i_n\to i$ $\P$-almost everywhere. \\ Then
$f_n\circ i_n\to f\circ i$ in $\P$-probability.
\end{lemma}
\begin{proof}
Let $d_Y$ denote the distance in $Y$. Up to replacing $d_Y$ by
$\min\{d_Y,1\}$, with no loss of generality we can assume that the
distance in $Y$ does not exceed 1. Fix $\eps>0$ and $g\in
C_b(X;Y)$ with $\int_Xd_Y(g,f)\,d\nu\leq\eps^2$. We have that
$\{d_Y(f_n\circ i_n,f\circ i)>3\eps\}$ is contained in $$
\{d_Y(f_n\circ i_n,g\circ i_n)>\eps\}\cup \{d_Y(g\circ i_n,g\circ
i)>\eps\}\cup \{d_Y(g\circ i,f\circ i)>\eps\}.
$$
The second set has infinitesimal $\P$-probability, since $g$ is
continuous and $i_n\to i$ $\P$-a.e.; the third set, by Markov
inequality, has $\P$-probability less than $\eps$; to estimate the
$\P$-probability of the first set we notice that
$$
\P(\{d_Y(f_n\circ i_n,g\circ
i_n)>\eps\})=\nu_n(\{d_Y(f_n,g)>\eps\})\leq
\frac{1}{\eps}\int_{X\times Y}\chi\,d(Id\times f_n)_\sharp\nu_n
$$
with $\chi(x,y):=d_Y(g(x),y)$. The weak convergence of $(Id\times
f_n)_\sharp\nu_n$ yields
\begin{eqnarray*}
\limsup_{n\to\infty}\P(\{d_Y(f_n\circ i_n,g\circ
i_n)>\eps\})&\leq& \frac{1}{\eps}\int_{X\times Y}\chi\,d(Id\times
f)_\sharp\nu\\&=& \frac{1}{\eps}\int_X
d_Y(g(x),f(x))\,d\nu(x)\leq\eps.
\end{eqnarray*}
\end{proof}

\section{Continuity equations and flows}

In this section we shall specify the basic assumptions on $\bb$ used
throughout this paper, and the conventions about \eqref{contieq}
concerning locally bounded respectively measure-valued solutions.
We shall also collect the basic definitions of
regular flows we shall work with, recalling first those used when
the state space is $\R^d$ and then extending these concepts to
$\Probabilities{\R^d}$.

\subsection{Continuity equations}

We consider a Borel vector field $\bb:[0,T]\times\R^d\to\R^d$, and
set $\bb_t(\cdot):=\bb(t,\cdot)$; we \emph{shall not} work with the
Lebesgue equivalence class of $\bb$, although a posteriori our
theory is independent of the choice of the representative (see
Remark~\ref{invariaflow}); this is important in view of the fact
that \eqref{contieq} involves possibly singular measures. Also, we
\emph{shall not} make any integrability assumption on $\bb$ besides
$L^1_{\rm loc}\bigl([0,T]\times\R^d\bigr)$ (namely, the Lebesgue
integral of $|\bb|$ is finite on $[0,T]\times B_R$ for all $R>0$);
the latter is needed in order to give a distributional sense to the
functional version of \eqref{contieq}, namely
\begin{equation}
\label{contieqw}
\frac{d}{dt} w_t+\nabla\cdot(\bb_t w_t)=0
\end{equation}
coupled with an initial condition $w_0=\bar w \in L_{\rm loc}^\infty(\R^d)$, when
$w_t$
is locally bounded in space-time.

It is well-known and easy to check that any distributional
solution $w(t,x)=w_t(x)$ to \eqref{contieqw} with $w_t$ locally
bounded in $\R^d$ uniformly in time, can be modified in a
$\Leb{1}$-negligible set of times in such a way that $t\mapsto
w_t$ is continuous w.r.t. the duality with $C_c(\R^d)$, and
well-defined limits exist at $t=0$, $t=T$ (see for instance
\cite[Lemma~8.1.2]{amgisa} for a detailed proof). In particular
the
initial condition $w_0=\bar w$ is then well defined, and we shall always
work with this weakly continuous representative.

In the sequel, we shall say that the continuity equation
\eqref{contieqw} has uniqueness in the cone of functions
$L^\infty_+\bigl([0,T];L^1\cap L^\infty(\R^d)\bigr)$ if, for any
$\bar w\in L^1\cap L^\infty(\R^d)$ nonnegative, there exists at most
one nonnegative solution $w_t$ to \eqref{contieqw} in
$L^\infty\bigl([0,T];L^1\cap L^\infty(\R^d)\bigr)$ satisfying the
condition
\begin{equation}\label{initial}
w_0=\bar w.
\end{equation}

Coming to
measure-valued solutions to
\eqref{contieq},
we say that $t\in
[0,T]\mapsto \mu_t\in {\mathscr M}_+(\R^d)$ solves \eqref{contieq} if
$|\bb|\in L^1_{\rm loc}\bigl((0,T)\times\R^d;\mu_tdt\bigr)$,
the equation holds in the sense of distributions and $t\mapsto\int\phi\,d\mu_t$
is continuous in $[0,T]$ for all $\phi\in C_c(\R^d)$.

\subsection{Flows in $\R^d$}

\begin{definition}[$\nu$-RLF in $\R^d$]\label{RLflow}
Let $\XX:[0,T]\times\R^d\to\R^d$ and $\nu\in {\mathscr M}_+(\R^d)$
with $\nu\ll\Leb{d}$ and with bounded density. We say that $\XX$ is
a $\nu$-RLF in $\R^d$ (relative to $\bb\in L^1_{\rm
loc}\bigl((0,T)\times\R^d\bigr)$) if the following two conditions
are fulfilled:
\begin{itemize}
\item[(i)] for $\nu$-a.e. $x$, the function $t\mapsto\XX(t,x)$
is an absolutely continuous integral solution to the ODE \eqref{ODE}
in $[0,T]$ with $\XX(0,x)=x$;
\item[(ii)] $\XX(t,\cdot)_\sharp\nu\leq C\Leb{d}$ for all $t\in
[0,T]$, for some constant $C$ independent of $t$.
\end{itemize}
\end{definition}

Notice that, in view of condition (ii), the assumption of bounded
density of $\nu$ is necessary for the existence of the $\nu$-RLF, as
$\XX(0,\cdot)_\sharp\nu=\nu$.

In this context, since all admissible initial measures $\nu$ are
bounded above by $C\Leb{d}$, uniqueness of the $\nu$-RLF can and
will be understood in the following stronger sense: if $f,\,g\in
L^1(\R^d)\cap L^\infty(\R^d)$ are nonnegative and $\XX$ and $\YY$
are respectively a $f\Leb{d}$-RLF and a $g\Leb{d}$-RLF, then
$\XX(\cdot,x)=\YY(\cdot,x)$ for $\Leb{d}$-a.e. $x\in\{f>0\}\cap
\{g>0\}$.

\begin{remark}[$BV$ vector fields]\label{rgoodb}
{\rm We shall use in particular the fact that the $\nu$-RLF exists
for all $\nu\leq C\Leb{d}$, and is unique, in the strong sense
described above, under the following
assumptions on $\bb$: $|\bb|$ is uniformly bounded, $\bb_t\in
BV_{\rm loc}(\R^d;\R^d)$ and $\nabla\cdot\bb_t=g_t\Leb{d}\ll\Leb{d}$
for $\Leb{1}$-a.e. $t\in (0,T)$, with
$$
\|g_t\|_{L^\infty(\R^d)}\in L^1(0,T),\qquad |D\bb_t|(B_R)\in
L^1(0,T)\quad\text{for all $R>0$,}
$$
where $|D\bb_t|$ denotes the total variation of the distributional
derivative of $\bb_t$. See \cite{ambrosio} or
\cite{cetraro} and the paper \cite{bouchut} for Hamiltonian vector
fields.}
\end{remark}

\begin{remark}[$\Leb{d}$-RLF] \label{rlebflow}{\rm In all situations where the $\nu$-RLF
exists and is unique, one can also define by an exhaustion procedure
a $\Leb{d}$-RLF $\XX$, uniquely determined (and well defined) by the
property
$$
\XX(\cdot,x)=\XX^f(\cdot,x)\quad\text{$\Leb{d}$-a.e. on $\{f>0\}$}
$$
for all $f\in L^\infty\cap L^1(\R^d)$ nonnegative, where $\XX^f$ is
the $f\Leb{d}$-flow. Also, it turns out that if \eqref{contieqw} has
\emph{backward} uniqueness, and if the constant $C$ in
Definition~\ref{RLflow}(ii) can be chosen independently of
$\nu\leq\Leb{d}$, then $\XX(t,\cdot)_\sharp\Leb{d}\leq C\Leb{d}$. We
don't prove this last statement here, since it will not be needed in
the rest of the paper, and we mention this just for completeness.
}\end{remark}

In the proof of stability and uniqueness results it is actually more
convenient to consider a generalized concept of flow, see
\cite{cetraro} for a more complete discussion. We denote the
evaluation map $(x,\omega)\in\R^d\times C([0,T];\R^d)\mapsto
\omega(t)\in\R^d$ again with $e_t$.

\begin{definition}[Generalized $\nu$-RLF in
$\R^d$]\label{GRLflow} Let $\nu\in {\mathscr M}_+(\R^d)$ and
$\eeta\in \Probabilities{\R^d\times C\bigl([0,T];\R^d\bigr)}$. We
say that $\eeta$ is a generalized $\nu$-RLF in $\R^d$ (relative to
$\bb$) if:
\begin{itemize}
\item[(i)] $(e_0)_\sharp\eeta=\nu$; \item[(ii)] $\eeta$ is
concentrated on the set of pairs $(x,\gamma)$, with $\gamma$
absolutely continuous solution to \eqref{ODE}, and $\gamma(0)=x$;
\item[(iii)] $(e_t)_\sharp\eeta\leq C\Leb{d}$ for all $t\in
[0,T]$, for some constant $C$ independent of $t$.
\end{itemize}
\end{definition}

\subsection{Flows in $\Probabilities{\R^d}$}

Given a nonnegative $\sigma$-finite measure
$\nnu\in\Measuresp{\Probabilities{\R^d}}$, we denote by
$\E\nnu\in\Measuresp{\R^d}$ its expectation, namely
$$
\int_{\R^d}\phi\,d\E\nnu=\int_{{\mathscr
P}(\R^d)}\int_{\R^d}\phi\,d\mu \,d\nnu(\mu)\qquad\text{for all
$\phi$ bounded Borel.}
$$

\begin{definition}[Regular measures on
$\Measuresp{\Probabilities{\R^d}}$]\label{RegMis} Let
$\nnu\in\Measuresp{\Probabilities{\R^d}}$. We say that $\nnu$ is
\emph{regular} if $\E\nnu\leq C\Leb{d}$ for some constant $C$.
\end{definition}

\begin{example}\label{eRegMis}
{\rm (1) The first standard example of a regular measure $\nnu$ is
the law under $\rho\Leb{d}$ of the map $x\mapsto\delta_x$, with
$\rho\in L^1(\R^{d})\cap L^\infty(\R^d)$ nonnegative. Actually, one can
even consider the law under $\Leb{d}$, and in this case
$\nnu$ would be $\sigma$-finite instead of a finite nonnegative
measure.

\smallskip
(2) If $d=2n$ and $z=(x,p)\in\R^n\times\R^n$ (this factorization
corresponds for instance to flows in a phase space), one may
consider the law under $\rho\Leb{n}$ of the map $x\mapsto
\delta_x\times\gamma$, with $\rho\in L^1(\R_x^{n})\cap
L^\infty(\R^n_x)$ nonnegative and $\gamma\in\Probabilities{\R^n_p}$
with $\gamma\le C\Leb{n}$; one can also choose $\gamma$ dependent on
$x$, provided $x\mapsto\gamma_x$ is measurable and $\gamma_x\leq
C\Leb{n}$ for some constant $C$ independent of $x$.

\smallskip
(3) We also expect that the \emph{entropic} measures built in
\cite{Sturm1}, \cite{Sturm2} are regular, see also the references
therein for more examples of ``natural'' reference measures on the
space of measures.}
\end{example}

As we explained in the introduction, Definition~\ref{RLflow} has a
natural (but not perfect) transposition to flows in
$\Probabilities{\R^d}$:

\begin{definition}[$\nnu$-RLF in $\Probabilities{\R^d}$] \label{RLflowmis} Let
$\mmu:[0,T]\times\Probabilities{\R^d}\to\Probabilities{\R^d}$ and
$\nnu\in\Measuresp{\Probabilities{\R^d}}$. We say that $\mmu$ is a
$\nnu$-RLF in $\Probabilities{\R^d}$ (relative to $\bb$ with
$|\bb|\in L^1_{\rm loc}\bigl((0,T)\times\R^d;\mu_tdt\bigr)$) if
\begin{itemize}
\item[(i)] for $\nnu$-a.e. $\mu$,
$t\mapsto\mu_t:=\mmu(t,\mu)$ is (weakly) continuous from $[0,T]$
to $\Probabilities{\R^d}$ with $\mmu(0,\mu)=\mu$ and $\mu_t$ solves
\eqref{contieq} in the sense of distributions;
\item[(ii)] $\E(\mmu(t,\cdot)_\sharp\nnu)\leq C\Leb{d}$ for all $t\in [0,T]$, for
some constant $C$ independent of $t$.
\end{itemize}
\end{definition}

Notice that no $\nnu$-RLF can exist if $\nnu$ is not regular, as
$\mmu(0,\cdot)_\sharp\nnu=\nnu$. Notice also that condition (ii) is
in some sense weaker than $\mmu(t,\cdot)_\sharp\nnu\leq C\nnu$
(which would be the analogue of (ii) in Definition~\ref{RLflow} if
we were allowed to choose $\nu=\Leb{d}$, see also
Remark~\ref{rlebflow}), but it is sufficient for our purposes. As a
matter of fact, because of infinite-dimensionality, the requirement
of quasi-invariance of $\nnu$ under the action of the flow $\mmu$
(namely the condition $\mmu(t,\cdot)_\sharp\nnu\ll\nnu$) would be a
quite strong condition: for instance, if the state space is a separable
Banach space $V$, the reference measure $\gamma$ is a nondegenerate
Gaussian measure, and $\bb(t,x)=v$, then $\XX(t,x)=x+tv$, and the
quasi-invariance occurs only if $v$ belongs to the Cameron-Martin
subspace $H$ of $V$, a dense but $\gamma$-negligible subspace. In
our framework, Example~\ref{eRegMis}(2) provides a natural measure
$\nnu$ that is not invariant, because its support is not invariant,
under the flow: to realize that invariance may fail,
it suffices to choose autonomous vector fields of
the form $\bb(x,p):=(p,-\nabla U(x))$.

\begin{remark}[Invariance of $\nnu$-RLF]\label{invariaflow}
{\rm Assume that $\mmu(t,\mu)$ is a $\nnu$-RLF relative to $\bb$
and $\tilde{\bb}$ is a modification of $\bb$, i.e., for
$\Leb{1}$-a.e. $t\in (0,T)$ the set
$N_t:=\{\bb_t\neq\tilde{\bb}_t\}$ is $\Leb{d}$-negligible. Then,
because of condition (ii) we know that, for all $t\in (0,T)$,
$\mmu(t,\mu)(N_t)=0$ for $\nnu$-a.e. $\mu$. By Fubini's theorem,
we obtain that, for $\nnu$-a.e. $\mu$, the set of times $t$ such
that $\mmu(t,\mu)(N_t)>0$ is $\Leb{1}$-negligible in $(0,T)$. As a
consequence $t \mapsto \mmu(t,\mu)$ is a solution to
\eqref{contieq} with $\tilde\bb_t$ in place of $\bb_t$, and $\mmu$
is a $\nnu$-RLF relative to $\tilde{\bb}$ as well. }\end{remark}

In the next definition, as in Definition~\ref{GRLflow}, we are going
to consider measures on $\Probabilities{\R^d}\times\Path{T}{\R^d}$,
the first factor being a convenient label for the initial position
of the path (an equivalent description could be given using just
measures on $\Path{T}{\R^d}$, at the price of an heavier use of
conditional probabilities, see \cite[Remark~11]{cetraro} for a more
precise discussion). We keep using the notation $e_t$ for the
evaluation map, so that $e_t(\mu,\omega)=\omega(t)$.

\begin{definition}[Generalized $\nnu$-RLF in
$\Probabilities{\R^d}$]\label{GRLflowmis} Let
$\nnu\in\Measuresp{\Probabilities{\R^d}}$ and $\eeta\in
\Measuresp{\Probabilities{\R^d}\times\Path{T}{\R^d}}$. We say that
$\eeta$ is a generalized $\nnu$-RLF in $\Probabilities{\R^d}$
(relative to $\bb$ with  $|\bb|\in L^1_{\rm
loc}\bigl((0,T)\times\R^d;\mu_tdt\bigr)$) if:
\begin{itemize}
\item[(i)] $(e_0)_\sharp\eeta=\nnu$;
\item[(ii)] $\eeta$ is
concentrated on the set of pairs $(\mu,\omega)$, with $\omega$
solving \eqref{contieq}, $\omega(0)=\mu$;
\item[(iii)] $\E((e_t)_\sharp\eeta)\leq C\Leb{d}$ for all $t\in [0,T]$,
for some constant $C$ independent of $t$.
\end{itemize}
\end{definition}

Again, by conditions (i) and (iii), no generalized $\nnu$-RLF can
exist if $\nnu$ is not regular. Of course any $\nnu$-RLF $\mmu$
induces a generalized $\nnu$-RLF $\eeta$: it suffices to define
\begin{equation}\label{induced}
\eeta:=(\Psi_\smmu)_\sharp\nnu,
\end{equation}
where
\begin{equation}\label{induced1}
\Psi_\smmu:\Probabilities{\R^d}\to
\Probabilities{\R^d}\times\Path{T}{\R^d},\qquad
\Psi_\smmu(\mu):=(\mu,\mmu(\cdot,\mu)).
\end{equation}
It turns out that existence results are stronger at the RLF level,
while results concerning uniqueness are stronger at the generalized
RLF level.

The transfer mechanisms between generalized and classical flows, and
between flows in $\Probabilities{\R^d}$ and flows in $\R^d$ are illustrated
by the next proposition.

\begin{proposition}\label{pindco}
Let $\eeta$ be a generalized $\nnu$-RLF in $\Probabilities{\R^d}$ relative to $\bb$. Then:
\begin{itemize}
\item[(i)] $\E\eeta$ is a generalized $\E\nnu$-RLF in $\R^d$ relative to $\bb$;
\item[(ii)] the measures
$\mu_t:=\E((e_t)_\sharp\eeta)=(e_t)_\sharp\E\eeta\in {\mathscr M}_+(\R^d)$ satisfy
\eqref{contieq}.
\end{itemize}
In addition, $\mu_t=w_t\Leb{d}$ with $w\in
L^\infty_+\bigl([0,T];L^1\cap L^\infty(\R^d)\bigr)$.
\end{proposition}
\begin{proof} Statement (i) is easy to prove, since the continuity equation is linear.
Statement (ii), namely that (single) time marginals of generalized flows in
$\R^d$ solve \eqref{contieq}, is proved in detail in
\cite[Page 8]{cetraro}. The final statement follows by the regularity
condition on $\eeta$.
\end{proof}

\section{Existence and uniqueness of regular Lagrangian flows}

In this section we recall the main existence and uniqueness results
of the $\nu$-RLF in $\R^d$, and see their extensions to $\nnu$-RLF
in $\Probabilities{\R^d}$. It turns out that existence and
uniqueness of solutions to \eqref{contieqw} in
$L^\infty_+\bigl([0,T];L^1\cap L^\infty(\R^d)\bigr)$ yields
existence and uniqueness of the $\nu$-RLF, and existence of this
flow implies existence of the $\nnu$-RLF when $\nnu$ is regular.
Also, the (apparently stronger) uniqueness of the $\nnu$-RLF is
still implied by the uniqueness of solutions to \eqref{contieqw} in
$L^\infty_+\bigl([0,T];L^1\cap L^\infty(\R^d)\bigr)$.

The following result is proved in \cite[Theorem~19]{cetraro} for the
part concerning existence and in \cite[Theorem~16,
Remark~17]{cetraro} for the part concerning uniqueness.

\begin{theorem}[Existence and uniqueness of the $\nu$-RLF in
$\R^d$]\label{texirlfrd} Assume that \eqref{contieqw} has existence
and uniqueness in $L^\infty_+\bigl([0,T];L^1\cap
L^\infty(\R^d)\bigr)$. Then, for all $\nu\in {\mathscr M}(\R^d)$
with $\nu\ll\Leb{d}$ and bounded density the $\nu$-RLF in $\R^d$
exists and is unique.
\end{theorem}

Now we can easily show that existence of the $\nu$-RLF implies
existence of the $\nnu$-RLF, by a superposition principle. However,
one might speculate that, for very rough vector fields, a $\nnu$-RLF
might exist in $\Probabilities{\R^d}$, not induced by any $\nu$-RLF
in $\R^d$.

\begin{theorem}[Existence of the $\nnu$-RLF in
$\Probabilities{\R^d}$]\label{texirlfprob} Let $\nu\in {\mathscr
M}(\R^d)$ with $\nu\ll\Leb{d}$ and bounded density, and assume that
a $\nu$-RLF $\XX$ in $\R^d$ exists. Then, for all
$\nnu\in\Measuresp{\Probabilities{\R^d}}$ with $\E\nnu=\nu$, a
$\nnu$-RLF $\mmu$ in $\Probabilities{\R^d}$ exists, and it is given
by
\begin{equation}\label{realRLF}
\mmu(t,\mu):=\int_{\R^d}\delta_{\sxX(t,x)}\,d\mu(x).
\end{equation}
\end{theorem}
\begin{proof}
The first part of property (i) in Definition~\ref{RLflowmis} is
obviously satisfied, since the fact that $t\mapsto\XX(t,x)$ solves
the ODE for some $x$ corresponds to the fact that
$t\mapsto\delta_{\sxX(t,x)}$ solves \eqref{contieq}. On the other
hand, since $\nnu$ is regular and $\XX$ is a RLF, we know that
$\XX(\cdot,x)$ solves the ODE for $\E\nnu$-a.e. $x$; it follows
that, for $\nnu$-a.e. $\mu$, $\XX(\cdot,x)$ solves the ODE for
$\mu$-almost every $x$, hence $\mmu(t,\mu)$ solves \eqref{contieq}
for $\nnu$-a.e. $\mu$. This proves (i).

Property (ii) follows by
\begin{eqnarray*}
\int_{\R^d} \phi(x)\,d\E(\mmu(t,\cdot)_\sharp\nnu)(x)&=&
\int_{{\mathscr P}(\R^d)}\int_{\R^d} \phi
d\mmu(t,\mu)\,d\nnu(\mu)\\&=& \int_{{\mathscr P}(\R^d)}\int_{\R^d}
\phi(\XX(t,x))\,d\mu(x)\,d\nnu(\mu)\\&=&
\int_{\R^d}\phi(\XX(t,x))\,d\nu(x)\leq CL\int_{\R^d}\phi(z)\,dz
\end{eqnarray*}
where $C$ is the same constant in Definition~\ref{RLflow}(ii) and
$L$ satisfies $\nu\leq L\Leb{d}$.
\end{proof}

The following lemma (a slight refinement of
\cite[Theorem~5.1]{ambrosio} and of \cite[Lemma~4.6]{amfiga})
provides a simple characterization of Dirac masses for measures on
$C_w\bigl([0,T];E\bigr)$ and for families of measures on $E$. Here
$E$ is a closed, convex and bounded subset of the dual of a
separable Banach space, endowed with a distance $d_E$ inducing the
weak$^*$ topology, so that $(E,d_E)$ is a compact metric space;
$C_w([0,T];E)$ denotes the space of continuous maps with values in
$(E,d_E)$, endowed with sup norm (so that these maps are continuous
with respect to the weak$^*$ topology). We shall apply this result
in the proof of Theorem \ref{tuniflow} with
\begin{equation}\label{defE}
E:=\left\{\mu\in\Measures{\R^d}:\ |\mu|(\R^d)\leq 1\right\}
\supset\Probabilities{\R^d},
\end{equation}
thought as a subset of $\bigl(C_0(\R^d)\bigr)^*$, where
$C_0(\R^d)$ denotes the set of continuous functions vanishing at
infinity (i.e. the closure of $C_c(\R^d)$ with respect to the
uniform convergence).

\begin{lemma}\label{simple}
Let $E\subset G^*$, with $G$ separable Banach space, be closed,
convex and bounded, and let $\ssigma$ be a positive finite measure
on $C_w\bigl([0,T];E\bigr)$. Then $\ssigma$ is a Dirac mass if and
only if $(e_t)_\sharp\ssigma$ is a
Dirac mass for all $t\in\Q\cap [0,T]$. \\
If $(F,{\cal F},\lambda)$ is a measure space, and a Borel family
$\{\nu_z\}_{z\in F}$ of probability measures on $E$ (i.e.
$z\mapsto\nu_z(A)$ is ${\cal F}$-measurable in $F$ for all $A\subset
E$ Borel) is given, then $\nu_z$ are Dirac masses for $\lambda$-a.e.
$z\in F$ if and only if for all $y\in G$ and $c\in\R$ there holds
\begin{equation}\label{orthomu}
\nu_z(\{x\in E:\ \langle x,y\rangle\leq c\})\nu_z(\{x\in E:\ \langle
x,y\rangle>c\})=0\quad\text{for $\lambda$-a.e. $z\in F$.}
\end{equation}
\end{lemma}
\begin{proof}
The first statement is a direct consequence of the fact that all
elements of $C_w\bigl([0,T];E\bigr)$ are weakly$^*$ continuous maps,
which are uniquely determined on $\Q\cap [0,T]$. In order to prove
the second statement, let us consider the sets $A_{ij}:=\{x\in E:\
\langle x,y_i\rangle\leq c_j\}$, where $y_i$ vary in a countable
dense set of $G$ and $c_j$ are an enumeration of the rational
numbers. By \eqref{orthomu} we obtain a $\lambda$-negligible set
$N_{ij}\in {\cal F}$ satisfying $\nu_z(A_{ij})\nu_z(E\setminus
A_{ij})=0$ for all $z\in F\setminus N_{ij}$. As a consequence, each
measure $\nu_z$, as $z$ varies in $F\setminus N_{ij}$, is either
concentrated on $A_{ij}$ or on its complement. For $z\in
F\setminus\cup_j N_{ij}$ it follows that the function
$x\mapsto\langle x,y_i\rangle$ is equivalent to a constant, up to
$\nu_z$-negligible sets. Since the functions $x\mapsto \langle
x,y_i\rangle $ separate points of $E$, $\nu_z$ is a Dirac mass for
all $z\in F\setminus\cup_{i,j}N_{ij}$ as desired.
\end{proof}

The next result shows that uniqueness of \eqref{contieq} in
$L^\infty_+\bigl([0,T];L^1\cap L^\infty(\R^d)\bigr)$ and existence
of a generalized $\nnu$-RLF imply existence of the $\nnu$-RLF and
uniqueness of both, the $\nnu$-RLF and the generalized $\nnu$-RLF.

\begin{theorem}[Existence and uniqueness of the $\nnu$-RLF in
$\Probabilities{\R^d}$]\label{tuniflow} Assume that \eqref{contieqw}
has uniqueness in $L^\infty_+\bigl([0,T];L^1\cap
L^\infty(\R^d)\bigr)$. If a generalized $\nnu$-RLF in
$\Probabilities{\R^d}$ $\eeta$ exists, then the $\nnu$-RLF $\mmu$ in
$\Probabilities{\R^d}$ exists. Moreover they are both unique, and
related as in \eqref{induced}, \eqref{induced1}.
\end{theorem}

\begin{proof} We fix a generalized $\nnu$-RLF $\eeta$ and we show
first that $\eeta$ is induced by a $\nnu$-RLF (this will prove in
particular the existence of the $\nnu$-RLF). To this end, denoting
by
$\pi:\Probabilities{\R^d}\times\Path{T}{\R^d}\to\Probabilities{\R^d}$
the projection on the first factor, we define by
$$
\eeta_\mu:=\E(\eeta|\pi=\mu)\in\Probabilities{\Path{T}{\R^d}}
$$
the induced conditional probabilities, so that
$d\eeta(\mu,\omega)=d\eeta_\mu(\omega)d\nnu(\mu)$. Taking into
account the first statement in Lemma~\ref{simple}, it suffices to
show that, for $\bar t\in\Q\cap [0,T]$ fixed, the measures $$
\theta_\mu:=\E((e_{\bar t})_\sharp\eeta|\omega(0)=\mu)=(e_{\bar
t})_\sharp\eeta_\mu\in\Measuresp{\Probabilities{\R^d}}$$ are Dirac
masses for $\nnu$-a.e. $\mu\in\Probabilities{\R^d}$. Still using
Lemma~\ref{simple}, we will check the validity of \eqref{orthomu}
with $\lambda=\nnu$. Since $\theta_\mu=\delta_\mu$ when $\bar
t=0$, we shall assume that $\bar t>0$.

Let us argue by contradiction, assuming the existence of
$L\in\BorelSets{\Probabilities{\R^d}}$ with $\nnu(L)>0$, $\phi\in
C_0(\R^d)$, $c\in\R$ such that both $\theta_\mu(A)$ and
$\theta_\mu\bigl(\Probabilities{\R^d}\setminus A\bigr)$ are strictly
positive for all $\mu\in L$, with
$$
A:=\left\{\rho\in\Probabilities{\R^d}:\ \int_{\R^d} \phi\,d\rho\leq
c\right\}.
$$
We will get a contradiction with the assumption that the equation
\eqref{contieqw} is well-posed in $L^\infty_+\bigl([0,T];L^1\cap
L^\infty(\R^d)\bigr)$, building two distinct nonnegative solutions
of the continuity equation with the same initial condition $\bar
w\in L^1\cap L^\infty(\R^d)$. With no loss of generality, possibly
passing to a smaller set $L$ still with positive $\nnu$-measure, we
can assume that the quotient
$g(\mu):=\theta_\mu(A)/\theta_\mu\bigl(\Probabilities{\R^d}\setminus
A\bigr)$ is uniformly bounded in $L$. Let
$\Omega_1\subset\Path{T}{\R^d}$ be the set of trajectories $\omega$
which belong to $A$ at time $\bar t$, and let $\Omega_2$ be its
complement; we can define positive finite measures $\eeta^i$,
$i=1,\,2$, in $\Probabilities{\R^d}\times\Path{T}{\R^d}$ by
$$
d\eeta^1(\mu,\omega):=d(\chi_{\Omega_1}\eeta_\mu)(\omega)
d(\chi_L\nnu)(\mu), \qquad
d\eeta^2(\mu,\omega):=d(\chi_{\Omega_2}\eeta_\mu)(\omega)d(\chi_L
g\nnu)(\mu).
$$
By Proposition~\ref{pindco}, both $\eeta^1$ and $\eeta^2$ induce
solutions $w^1_t,\,w^2_t$ to the continuity equation which are
uniformly bounded (just by comparison with the one induced by
$\eeta$) in space and time. Moreover, since
$$
(e_0)_\sharp\eeta^1=\theta_\mu(A)\chi_L(\mu)\nnu
$$
and analogously
$$
(e_0)_\sharp\eeta^2=\theta_\mu\bigl(\Probabilities{\R^d}\setminus
A\bigr)\chi_L(\mu)g(\mu)\nnu,
$$
our definition of $g$ gives that
$(e_0)_\sharp\eeta^1=(e_0)_\sharp\eeta^2$. Hence, both solutions
$w^1_t$, $w^2_t$ start from the same initial condition $\bar w(x)$,
namely the density of $\E(\theta_\mu(A)\chi_L(\mu)\nnu)$ with
respect to $\Leb{d}$. On the other hand, it turns out that
\begin{eqnarray*}
\int_{\R^d}\phi w^1_{\bar
t}\,dx&=&\int_L\int_{\Omega_1}\int_{\R^d}\phi\,d\omega(\bar{t})\,
d\eeta_\mu(\omega)\,d\nnu(\mu)\\
&=&\int_L\int_{\Path{T}{\R^d}}\chi_A(\omega(\bar{t}))
\int_{\R^d}\phi\,d\omega(\bar{t})\, d\eeta_\mu(\omega)\,d\nnu(\mu)
\\&=&\int_L\int_A\int_{\R^d}\phi\,d\rho\,d\theta_\mu(\rho)\,d\nnu(\mu)\leq
c\int_L\theta_\mu(A)d\nnu(\mu).
\end{eqnarray*}
Analogously, we have
$$\int_{\R^d}\phi w^2_{\bar t}\,dx>c\int_L\theta_\mu
\bigl(\Probabilities{\R^d}\setminus A\bigr)g(\mu)\,d\nnu(\mu)=
c\int_L\theta_\mu(A)\,d\nnu(\mu).
$$
Therefore $w^1_{\bar t}\neq w^2_{\bar t}$ and uniqueness of the
continuity equation is violated.

Now we can prove uniqueness: if $\ssigma$ is any other generalized
$\nnu$-RLF, we know $\ssigma$ is induced by a $\nnu$-RLF, hence for
$\nnu$-a.e. $\mu$ also the measures $\E(\ssigma|\omega(0)=\mu)$ are
Dirac masses; but, since the property of being a generalized flow is
stable under convex combinations, also the measures (corresponding
to the generalized $\nnu$-RLF $(\eeta+\ssigma)/2$)
$$
\frac{1}{2}\E(\eeta|\omega(0)=\mu)+
\frac{1}{2}\E(\ssigma|\omega(0)=\mu)=
\E\biggl(\frac{\eeta+\ssigma}{2}|\omega(0)=\mu\biggr)
$$
must be Dirac masses for $\nnu$-a.e. $\mu$. This can happen only if
$\E(\eeta|\omega(0)=\mu)=\E(\ssigma|\omega(0)=\mu)$ for $\nnu$-a.e.
$\mu$, hence $\ssigma=\eeta$. Finally, since distinct $\nnu$-RLF
$\mmu$ and $\mmu'$ induce distinct generalized $\nnu$-RLF $\eeta$
and $\eeta'$, uniqueness is proved also for $\nnu$-RLF.
\end{proof}

\section{Stability of the $\nnu$-RLF in
$\Probabilities{\R^d}$}\label{sstable}

In the statement of the stability result we shall consider varying
measures $\nnu_n\in\Probabilities{\Probabilities{\R^d}}$, $n\geq 1$,
and a limit measure $\nnu$. (The assumption that all $\nnu_n$ are
probability measures is made in order to avoid technicalities which
would obscure the main ideas behind our stability result, and one
can always reduce to this case by renormalizing the measures.
Moreover, in the applications we have in mind, our measures $\nnu_n$
will always have unitary total mass.) We shall assume that the
$\nnu_n$ are generated as $(i_n)_\sharp\P$, where $(W,{\mathcal
F},\P)$ is a probability measure space and
$i_n:W\to\Probabilities{\R^d}$ are measurable; accordingly, we shall
also assume that $\nnu=i_\sharp\P$, with $i_n\to i$ $\P$-almost
everywhere. These assumptions are satisfied in the applications we
have in mind, and in any case Skorokhod's theorem (see \cite[\S8.5,
Vol. II]{bogachevII}) could be used to show that weak convergence of
$\nnu_n$ to $\nnu$ always implies this sort of representation, even
with $W=[0,1]$ endowed with the standard measure structure, for
suitable $i_n,\,i$.

Many formulations of the stability result are indeed possible and we
have chosen one specific for the application we have in mind.
Henceforth we fix an autonomous vector field $\bb:\R^d\to\R^d$
satisfying the following regularity conditions:

\begin{itemize}
\item[(a)] $d=2n$ and $\bb(x,p)=(p,\cc(x))$, $(x,p)\in\R^d$,
$\cc:\R^n\to\R^n$ Borel and locally integrable;
\item[(b)] there
exists a closed $\Leb{n}$-negligible set $S$ such that $\cc$ is
locally bounded on $\R^n\setminus S$;
\item[(c)] the discontinuity set $\Sigma$ of $\cc$ is
$\Leb{n}$-negligible.
\end{itemize}

\begin{lemma}\label{lremsing}
Let $S\subset\R^n$ closed, and assume that $\bb$ is representable as
in (a) above. Let $\mu_t:[0,T]\to\Probabilities{\R^d}$ be solving
\eqref{contieq} in the sense of distributions in $(\R^n\setminus
S)\times\R^n$ and assume that
$$
\int_0^T\int_{B_R}\frac{1}{{\rm
dist}^\beta(x,S)}\,d\mu_t(x,p)dt<\infty\qquad\forall\, R>0
$$
for some $\beta>1$ (with the convention $1/0=+\infty$). Then
\eqref{contieq} holds in the sense of distributions in $\R^d$.
\end{lemma}
\begin{proof}
First of all, the assumption implies that $\mu_t(S\times\R^n)=0$ for
$\Leb{1}$-a.e. $t\in (0,T)$. The proof of the global validity of the
continuity equation uses the classical argument of removing the
singularity by multiplying any test function $\phi\in
C^\infty_c(\R^d)$ by $\chi_k$, where $\chi_k(x)=\chi(k{\rm
dist}(x,S))$ and $\chi$ is a smooth cut-off function equal to $0$ on
$[0,1]$ and equal to $1$ on $[2,+\infty)$, with $0\leq\chi'\leq 2$.
If we use $\phi\chi_k$ as a test function, since $\chi_k$ depends on
$x$ only, we can use the particular structure (a) of $\bb$ to write
the term depending on the derivatives of $\chi_k$ as
$$
k\int_0^T\int_{\R^d}\phi\chi'(k{\rm dist}(x,S))\langle p,\nabla{\rm
dist}(x,S)\rangle\,d\mu_t(x,p)dt.
$$
If $K$ is the support of $\phi$, the integral above can be bounded by
$$
2\max_K|p\phi|\int_0^T\int_{\{x\in K: k{\rm dist}(x,S)\leq
2\}}k\,d\mu_t(x,p)dt\leq\frac{2^{\beta+1}\max_K|p\phi|}{k^{\beta-1}}
\int_0^T\int_K\frac{1}{{\rm dist}^\beta(x,S)}\,d\mu_t(x,p)dt
$$
and as $\beta>1$ the right hand side is infinitesimal as $k\to\infty$.
\end{proof}

The following stability result is adapted to the application we have
in mind: we shall apply it to the case when $\mmu_n(t,\mu)$ are
Husimi transforms of wavefunctions.

\begin{theorem}[Stability of the $\nnu$-RLF in
$\Probabilities{\R^d}$]\label{tstable} Let $i_n,\,i$ be as above and
let $\mmu_n:[0,T]\times i_n(W)\to\Probabilities{\R^d}$ be satisfying
$\mmu_n(0,i_n(w))=i_n(w)$ and the following conditions:
\begin{itemize}
\item[(i)] (asymptotic regularity)
$$\limsup_{n\to\infty}\int_W\int_{\R^d}\phi\,d\mmu_n(t,i_n(w))\,d\P(w)\leq
C\int_{\R^d} \phi\,dx
$$
for all $\phi\in C_c(\R^d)$ nonnegative, for some constant $C$
independent of $t$;
\item[(ii)] (uniform decay away from the singularity) for some $\beta>1$
\begin{equation}\label{ali4}
\sup_{\delta>0}\limsup_{n\to\infty}\int_W\int_0^T
\int_{B_R}\frac{1}{{\rm
dist}^\beta(x,S)+\delta}\,d\mmu_n(t,i_n(w))\,dt \,d\P(w)<\infty
\qquad\forall\, R>0;
\end{equation}
\item[(iii)] (space tightness) for all $\delta>0$,
$\P\Bigl(\bigl\{w\in W:\ \sup\limits_{t\in
[0,T]}\mmu_n(t,i_n(w))(\R^d\setminus B_R)>\delta\bigr\}\Bigr)\to 0$
as $R\to\infty$ uniformly in $n$;
\item[(iv)] (time tightness) for $\P$-a.e. $w\in
W$, for all $n\geq 1$ and $\phi\in C^\infty_c(\R^d)$,
$t\mapsto\int_{\R^d}\phi\,d\mmu_n(t,i_n(w))$ is absolutely
continuous in $[0,T]$ and, uniformly in $n$,
$$
\lim_{M\uparrow\infty}\P\biggl(\Bigl\{w\in W:\
\int_0^T\biggl|\biggr(\int_{\R^d}\phi\,d\mmu_n(t,i_n(w))\biggr)'\biggr|\,dt>M\Bigr\}\biggr)=0;
$$
\item[(v)] (limit continuity equation)
\begin{equation}\label{ali6}
\lim_{n\to\infty}\int_W
\biggl|\int_0^T\biggl[\varphi'(t)\int_{\R^d}\phi\,d\mmu_n(t,i_n(w))+\varphi(t)\int_{\R^d}
\langle\bb,\nabla\phi\rangle\,d\mmu_n(t,i_n(w))\biggr]\,dt\biggr|\,d\P(w)=0
\end{equation}
for all $\phi\in C^\infty_c\bigl(\R^d\setminus(S\times\R^n)\bigr)$,
$\varphi\in C^\infty_c(0,T)$.
\end{itemize}
Assume, besides (a), (b), (c) above, that \eqref{contieqw} has
uniqueness in $L^\infty_+\bigl([0,T];L^1\cap L^\infty(\R^d)\bigr)$.
Then the $\nnu$-RLF $\mmu(t,\mu)$ relative to $\bb$ exists, is
unique, and
\begin{equation}\label{cetraro1}
\lim_{n\to\infty} \int_W\sup_{t\in [0,T]}d_{{\mathscr
P}}(\mmu_n(t,i_n(w)),\mmu(t,i(w)))\,d\P(w)=0.
\end{equation}
\end{theorem}

\begin{proof} Let
$(\eeta_n)\subset\Measuresp{\Probabilities{\R^d}\times\Path{T}{\R^d}}$
be induced by $\mmu_n$ pushing forward $\nnu_n=(i_n)_\sharp\P$ via the map
$\mu\mapsto(\mu,\mmu_n(t,\mu))$. Conditions (iii) and (iv)
correspond, respectively, to conditions (i) and (ii) of
Proposition~\ref{ptightness}, hence the marginals of $\eeta_n$ on
$\Path{T}{\R^d}$ are tight; since the first marginals, namely
$\nnu_n$, are tight as well, a simple tightness criterion in product
spaces (see for instance \cite[Lemma~5.2.2]{amgisa}) gives that
$(\eeta_n)$ is tight. We consider a weak limit point $\eeta$ of
$(\eeta_n)$ and prove that $\eeta$ is the unique generalized
$\nnu$-RLF relative to $\bb$; this will give that the whole sequence
$(\eeta_n)$ weakly converges to $\eeta$. Just to simplify notation,
we assume that the whole sequence $(\eeta_n)$ weakly converges to
$\eeta$.

We check conditions (i), (ii), (iii) of
Definition~\ref{GRLflowmis}. First, since $\mmu_n(0,\mu)=\mu$
$\nnu_n$-a.e., we get $(e_0)_\sharp\eeta_n=\nnu_n$, hence
$(e_0)_\sharp\eeta=\nnu$ and condition (i) is satisfied. Second,
we check condition (iii): for $\phi\in C_c(\R^d)$ nonnegative we
have
$$
\int_{\R^d}\phi\,d\E((e_t)_\sharp\eeta_n)=\int_{\R^d}\phi\,
d\E(\mmu(t,\cdot)_\sharp\nnu_n)=
\int_W\int_{\R^d}\phi\,d\mmu_n(t,i_n(w))\,d\P(w)
$$
and we can use assumption (i) to conclude that
\begin{equation}\label{ali5}
\int_{\R^d}\phi\,d\E((e_t)_\sharp\eeta)\leq C\int_{\R^d}\phi\,dz
\qquad\forall\, t\in [0,T],
\end{equation}
so that condition (iii) is fulfilled.

Finally we check condition (ii). Since $\eeta_n$ are concentrated on
the closed set of pairs $(\mu,\omega)$ with $\omega(0)=\mu$, the
same is true for $\eeta$; it remains to show that $\omega(t)$ solves
\eqref{contieq} for $\eeta$-a.e. $(\mu,\omega)$. We shall denote by
$\ssigma\in\Measuresp{\Path{T}{\R^d}}$ the projection of $\eeta$ on
the second factor and prove that \eqref{contieq} holds for
$\ssigma$-a.e. $\omega$.

We fix $\phi\in C^\infty_c\bigl(\R^d\setminus (S\times\R^n)\bigr)$
and $\varphi\in C^\infty_c(0,T)$; we claim that the discontinuity
set of the bounded map
\begin{equation}\label{ali8}
\omega\mapsto
\int_0^T\biggl[\varphi'(t)\int_{\R^d}\phi\,d\omega(t)+\varphi(t)\int_{\R^d}
\langle\bb,\nabla\phi\rangle\,d\omega(t)\biggr]\,dt
\end{equation}
is $\ssigma$-negligible. Indeed, using \eqref{finerco} with $X=\R^d$
this discontinuity set is easily seen to be contained in
\begin{equation}\label{ali66}
\left\{\omega\in\Path{T}{\R^d}:\
\int_0^T\omega(t)(\Sigma\times\R^n)\,dt>0\right\},
\end{equation}
where $\Sigma$ is the discontinuity set of $\cc$. Since
$\Leb{d}(\Sigma\times\R^n)=0$, by assumption (c), for all
$t\in [0,T]$ the inequality \eqref{ali5} gives
$\omega(t)(\Sigma\times\R^n)=0$ for $\ssigma$-a.e. $\omega$; by
Fubini's theorem in $[0,T]\times\Path{T}{\R^d}$
we obtain that the set in \eqref{ali66} is
$\ssigma$-negligible.

Now we write assumption \eqref{ali4} in terms of $\eeta_n$ as
$$
\sup_{\delta>0}\limsup_{n\to\infty} \int_{{\mathscr
P}(\R^d)\times\Path{T}{\R^d}}\int_0^T \int_{B_R}\frac{1}{{\rm
dist}^\beta(x,S)+\delta}\,d\omega(t)\,dt
\,d\eeta_n(\mu,\omega)<\infty \qquad\forall \,R>0,
$$
and take the limit thanks to Fatou's Lemma and the Monotone
Convergence Theorem to obtain
\begin{equation}\label{paola}
\int_{\Path{T}{\R^d}}\int_0^T
\int_{B_R}\frac{1}{{\rm dist}^\beta(x,S)}\,d\omega(t)\,dt
\,d\ssigma(\omega)<\infty \qquad\forall\, R>0.
\end{equation}

Next we write assumption (v) in terms of $\eeta_n$ as
$$
\lim_{n\to\infty} \int_{{\mathscr P}(\R^d)\times\Path{T}{\R^d}}
\zeta\biggl|\int_0^T\biggl[\varphi'(t)\int_{\R^d}\phi\,d\omega(t)+
\varphi(t)\int_{\R^d}\langle\bb,\nabla\phi\rangle\,d\omega(t)\biggr]
\,dt\biggr|\,d\eeta_n(\mu,\omega)=0
$$
with $\z\in C_b\bigl(\Probabilities{\R^d}\times\Path{T}{\R^d}\bigr)$
nonnegative; then, the claim on the continuity of the map in
\eqref{ali8} and \eqref{finerco} with
$X=\Probabilities{\R^d}\times\Path{T}{\R^d}$ allow to conclude that
$$
\int_{{\mathscr P}(\R^d)\times\Path{T}{\R^d}}
\zeta\biggl|\int_0^T\biggl[\varphi'(t)\int_{\R^d}\phi\,d\omega(t)+
\varphi(t)\int_{\R^d}\langle\bb,\nabla\phi\rangle\,d\omega(t)\biggr]
\,dt\biggr|\,d\eeta(\mu,\omega)=0.
$$
Now we fix ${\cal A}\subset C^\infty_c\bigl(\R^d\setminus
(S\times\R^n)\bigr)$, ${\cal B}\subset C^\infty_c(0,T)$ countable
dense, and use the fact that $\z$ is arbitrary to find a
$\ssigma$-negligible set
$N\subset\Path{T}{\R^d}$ such that
$$
\int_0^T\biggl[\varphi'(t)\int_{\R^d}\phi\,d\omega(t)+
\varphi(t)\int_{\R^d}\langle\bb_t,\nabla\phi\rangle\,d\omega(t)\biggr]\,dt=0
\qquad\forall\,\phi\in{\mathcal
A},\,\,\forall\,\varphi\in{\mathcal B}
$$
for all $\omega\notin N$, and by a density argument we conclude that
$\ssigma$ is concentrated on solutions to the continuity equation in
$\R^d\setminus (S\times\R^n)$. By Lemma~\ref{lremsing} and
\eqref{paola} we obtain that $\ssigma$-a.e. the continuity equation
holds globally.

By Theorem~\ref{tuniflow} we know that the $\nnu$-RLF
$\mmu(t,\mu)$ in $\Probabilities{\R^d}$ exists, is unique, and
related to the unique generalized $\nnu$-RLF $\eeta$ as in
\eqref{induced}, \eqref{induced1}. This proves that we have
convergence of the whole sequence $(\eeta_n)$ to $\eeta$. By
applying Lemma~\ref{lplans} with $X=\Probabilities{\R^d}$ and
$Y=\Path{T}{\R^d}$ we conclude that \eqref{cetraro1} holds.
\end{proof}

In the next remark we consider some extensions of this result to the
case when $\bb$ satisfies (a), (b) only, so that no information is
available on the discontinuity set $\Sigma$ of $\cc$.

\begin{remark}\label{rven}
{\rm Assume that $\bb$ satisfies (a), (b) only. Then the conclusion
of Theorem~\ref{tstable} is still valid, provided the asymptotic
regularity condition (i) holds in a stronger form, namely
$$
\int_W\int_{\R^d}\phi\,d\mmu_n(t,i_n(w))\,d\P(w)\leq C\int_{\R^d}
\phi\,dx\qquad\forall\,\phi\in C_c(\R^d),\,\,\phi\geq 0,\,\,n\geq 1
$$
for some constant $C$ independent of $t$. Indeed, assumption (c)
was needed only to pass to the limit, in the weak convergence of
$\eeta_n$ to $\eeta$, with test functions of the form
\eqref{ali8}. But, if the stronger regularity condition above
holds, convergence always holds by a density argument: first one
checks this with $\bb$ continuous and bounded on $\supp\phi$, and
in this case the test function is continuous and bounded; then one
approximates $\bb$ in $L^1$ on $\supp\phi$ by bounded continuous
functions.}\end{remark}

\section{Well-posedness of the continuity equation with a singular
potential}\label{sgoodliou}

In this section we shall assume that $d=2n$ and consider a more particular
class of autonomous and Hamiltonian vector fields $\bb:\R^d\to\R^d$ of the
form
$$
\bb(z)=\bigl(p,-\nabla U(x)\bigr),\qquad z=(x,p)\in\R^n\times\R^n.
$$
Having in mind the application to the convergence of the
Wigner/Husimi transforms in quantum molecular dynamics, we assume
that:
\begin{itemize}
\item[(i)] there exists a closed $\Leb{n}$-negligible set
$S\subset\R^n$ such that $U$ is locally Lipschitz in
$\R^n\setminus S$ and $\nabla U\in BV_{\rm
loc}(\R^n\setminus S;\R^n)$; \item[(ii)] $U(x)\to +\infty$ as $x\to S$.
\item[(iii)] $U$ satisfies
\begin{equation}
\label{eq:local bddness} \esssup_{U(x) \leq M} \frac{|\n
U(x)|}{1+|x|} <\infty \qquad \forall\, M \geq 0.
\end{equation}
\end{itemize}

\begin{theorem} \label{tmain3} Under assumptions (i), (ii), (iii),
the continuity equation \eqref{contieqw} has existence and
uniqueness in $L_+^\infty\bigl([0,T];L^1\cap L^\infty(\R^d)\bigr)$.
\end{theorem}
\begin{proof} (Uniqueness) Let $w_t\in L_+^\infty\bigl([0,T];L^1\cap
L^\infty(\R^d)\bigr)$ be a solution to \eqref{contieqw}, and
consider a smooth compactly supported function $\phi:\R \to \R^+$.
Set $E=E(x,p):=\frac{1}{2}|p|^2+U(x)$. Then, since $U$ is locally
Lipschitz on sublevels $\{U \leq \ell\}$ for any $\ell \in \R$ (by
(i)-(ii)), $\phi\circ E$ is uniformly bounded and locally Lipschitz
in $\R^d$. Moreover
$$
\langle\n\bigl(\phi\circ
E\bigr)(z),\bb(z)\rangle=\phi'(E(z))\langle\n E(z),\bb(z)\rangle=0
\qquad \text{for $\Leb{d}$-a.e. $z\in\R^d$},
$$
and we easily deduce that also $(\phi\circ E)w_t\in
L_+^\infty\bigl([0,T];L^1\cap L^\infty(\R^d)\bigr)$ solves
\eqref{contieqw}. Let $M>0$ be large enough so that
$\supp\phi\subset [-M,M]$, and let $\psi:\R \to \R^+$ be a smooth
cut-off function such that $\psi\equiv 1$ on $[-M,M]$. Then $\phi
\circ E=(\psi \circ E) (\phi\circ E)$, which implies that
$(\phi\circ E)w_t$ solves \eqref{contieqw} with the vector field
$\tilde \bb:=(\psi\circ E)\bb$. Now, thanks to (i)-(iii), it is
easily seen that the following properties hold:
\begin{equation}
\label{eq:bound tilde b} \tilde\bb\in BV_{{\rm
loc}}(\R^d;\R^d),\qquad\esssup\frac{|\tilde \bb|(z)}{1+|z|} <\infty.
\end{equation}
Indeed, the first one is a direct consequence of (i)-(ii), while
the second one follows from (ii)-(iii) and the simple estimate
$$
\esssup\limits_{E(z) \leq M'}\frac{|\bb(z)|}{1+|z|} \leq \biggl(\sup
\frac{|p|}{1+|p|}\biggr) + \biggl(\esssup\limits_{U(x) \leq
M'}\frac{|\n U(x)|}{1+|x|}\biggr) <\infty\qquad \forall\,M'>0.
$$
Thanks to \eqref{eq:bound tilde b}, we can apply \cite[Theorems 34
and 26]{cetraro} to deduce that $(\phi\circ E)w_t$ is unique,
given the initial condition $\mu_0=(\phi\circ E)w_0\Leb{d}$. Since
$E(z)$ is finite for $\Leb{d}$-a.e. $z$, by the arbitrariness of
$\phi$ we easily obtain that $w_t$ is unique, given the initial
condition $w_0$.

\noindent (Existence) We now want to prove existence of solutions in
$L_+^\infty\bigl([0,T];L^1\cap L^\infty(\R^d)\bigr)$. Let $\bar w\in
L^1\cap L^\infty(\R^d)$ be nonnegative and let us consider a
sequence of smooth globally Lipschitz functions $V_k$ with $|\nabla
V_k-\nabla U|\to 0$ in $L^1_{\rm loc}(\R^n)$; standard results imply
the existence of nonnegative solutions $w^k$ to the continuity
equation with velocity $\bb^k:=(p,-\nabla V_k)$ with $w^k_0=\bar w$,
$\int_{\R^d} w_t^k\,dx\,dp=\int_{\R^d} w_0^k\,dx\,dp$ and with
$\|w_t^k\|_\infty\leq\|w_0^k\|_\infty$ (they are the push forward of
$w^k_0$ under the flow map of $\bb^k$). Since $\phi\mapsto
\int_{\R^d}w^k_t\phi\,dx\,dp$ are equi-continuous for all $\phi\in
C_c^1(\R^d)$, we can assume the existence of $w\in
L_+^\infty\bigl([0,T];L^1\cap L^\infty(\R^d)\bigr)$ with $w^k_t\to
w_t$ weakly, in the duality with $C^1_c(\R^d)$, for all $t\geq 0$.
Taking the limit as $k\to\infty$ immediately gives that $w_t$ is a
solution to \eqref{contieqw}.
\end{proof}

\begin{theorem}\label{tmain1}
Under assumptions (i), (ii), (iii), the $\nu$-RLF
$\bigl(\xx(t,x,p),\pp(t,x,p)\bigr)$ in $\R^{2n}$ and the $\nnu$-RLF
$\mmu(t,\mu)$ in $\Probabilities{\R^{2n}}$ relative to
$\bb(x,p):=(p,-\nabla U(x))$ exist and are unique. They are related
by
\begin{equation}\label{svitla}
\mmu(t,\mu)=\int_{\R^{2d}}
\delta_{(\sxx(t,x,p),\spp(t,x,p))}\,d\mu(x,p).
\end{equation}
\end{theorem}
\begin{proof} Existence and uniqueness of the $\nu$-RLF in $\R^d$
follow by Theorem~\ref{tmain3} and Theorem~\ref{texirlfrd}. The
uniqueness of the $\nnu$-RLF in $\Probabilities{\R^d}$ and its
relation with the $\nu$-RLF are a consequence respectively of
Theorem~\ref{tuniflow} and Theorem~\ref{texirlfprob}.
\end{proof}

\section{Estimates on solutions to \eqref{eq:Sch} and on error
terms}\label{estimates}

In this section we collect some a-priori estimates on solutions to
\eqref{eq:Sch} and on the error terms $\Errordelta{U}{\psi}$,
$\Error{U}{\psi}$, appearing respectively in \eqref{Wigner_calc_in}
and \eqref{Wigner_calc}.

We recall that the Husimi transform $\psi\mapsto\tilde{W}_\e\psi$
can be defined in terms of convolution of the Wigner transform with
the $2n$-dimensional Gaussian kernel with variance $\e/2$
\begin{equation}\label{heatK}
G_\e^{(2n)}(x,p):=
\frac{e^{-(|x|^2+|p|^2)/\e}}{(\pi\e)^n}=G^{(n)}_\e(x)G^{(n)}_\e(p),
\end{equation}
namely $\tilde W_\e\psi=(W_\e\psi)\ast G^{(2n)}_\e$. It turns out
that the asymptotic behaviour as $\e\to 0$ is the same for the
Wigner and the Husimi transform (see also \eqref{uniformly} below
for a more precise statement).

For later use, we recall that the $x$ marginal of $W_\e\psi$ is the
position density $|\psi|^2\Leb{n}$. Also, the change of variables
\begin{equation}\label{change}
\begin{cases}
x+\frac{\e}{2}y=u\\
x-\frac{\e}{2}y=u'
\end{cases}
\end{equation}
and a simple computation show that the $p$ marginal of $W_\e\psi$ is
the momentum density, namely $(2\pi\e)^{-n}|{\cal
F}\psi|^2(p/\e)\Leb{n}$ (strictly speaking these identities are only
true in the sense of principal values, since $W_\e\psi$, despite
tending to zero as $|(x,p)|\to\infty$, does not in general belong to
$L^1$). Since the Gaussian kernel $G_\e^{(2n)}(x,p)$ in
\eqref{heatK} has a product structure, it turns out that
\begin{equation}\label{compamarginals1}
\int_{\R^n}\tilde
W_\e\psi(x,p)\,dp=\int_{\R^n}|\psi|^2(x-x')G_\e^{(n)}(x')\,dx',
\end{equation}
\begin{equation}\label{compamarginals2}
\int_{\R^n}\tilde
W_\e\psi(x,p)\,dx=\Bigl(\frac{1}{2\pi\e}\Bigr)^{n}\int_{\R^n}|{\cal
F}\psi|^2\bigl(\frac{p-p'}{\e}\bigr)G_\e^{(n)}(p')\,dp'.
\end{equation}
Since $\tilde W_\e\psi$ is nonnegative (see Section~\ref{sHusimi}
for details) the two identities above hold in the standard sense.

As in \cite{lionspaul} we shall consider the completion ${\cal A}$
of $C^\infty_c(\R^{2n})$ with respect to the norm
\begin{equation}\label{defA}
\|\varphi\|_{{\cal A}}:=\int_{\R^n}\sup_{x\in\R^n}|{\cal
F}_p\varphi|(x,y)\,dy\qquad\varphi\in C^\infty_c(\R^{2n}),
\end{equation}
where ${\cal F}_p$ denotes the partial Fourier transform with
respect to $p$. It is easily seen that
$\sup|\varphi|\leq\|\phi\|_{{\cal A}}$, hence ${\cal A}$ is
contained in $C_b(\R^{2n})$ and ${\mathscr M}(\R^{2n})$ canonically
embeds into ${\cal A}'$ (the embedding is injective by the density
of $C^\infty_c(\R^{2n})$). The norm of ${\cal A}$ is technically
convenient because of the simple estimate
\begin{equation}\label{lionspaul}
\biggl|\int_{\R^{2n}}\varphi
W_\e\psi\,dxdp\biggr|\leq\frac1{(2\pi)^n} \|\varphi\|_{{\cal A}}
\|\psi\|^2_2.
\end{equation}
Since for all $\varphi\in C^\infty_c(\R^{2n})$ one has $\varphi\ast
G_\e^{(2n)}\to\varphi$ in ${\cal A}$ as $\e\downarrow 0$, it follows
that
\begin{equation}\label{uniformly}
\lim_{\e\downarrow 0}\int_{\R^d} \varphi
W_\e\psi\,dx\,dp-\int_{\R^d}\varphi\tilde{W}_\e\psi\,dx\,dp=0
\qquad\text{uniformly on bounded subsets of $L^2(\R^d;\C)$.}
\end{equation}
This will obviously be an ingredient in transferring the dynamical properties from the
Wigner to the Husimi transforms.

\subsection{The PDE satisfied by the Husimi transforms}

In this short section we see how \eqref{Wigner_calc_in} is modified
in passing from the Wigner to the Husimi transform. Denoting by
$\tau_{(y,q)}$ the translation in phase space induced by
$(y,q)\in\R^n\times\R^n$, from \eqref{Wigner_calc_in} we get
$$
\partial_t\tau_{(y,q)}W_\e\psi^\e_t+(p-q)\cdot\nabla_x\tau_{(y,q)}W_\e\psi^\e_t=
\tau_{(y,q)}\Errordelta{U}{\psi^\e_t}
$$
in the sense of distributions. Since $\tilde W_\e\psi^\e_t$ is an
average of translates of $W_\e\psi^\e_t$, we get (still in the sense
of distributions)
\begin{equation}\label{pdehusimi}
\partial_t\tilde W_\e\psi^\e_t+p\cdot\nabla_x\tilde W_\e\psi^\e_t=
\Errordelta{U}{\psi^\e_t}\ast G^{(2n)}_\e+\sqrt{\e}\nabla_x\cdot
[W_\e\psi^\e_t\ast \bar{G}_\e^{(2n)}],
\end{equation}
where
\begin{equation}\label{pdehusimibis}
\bar{G}_\e^{(2n)}(y,q):=\frac{q}{\sqrt{\e}}G^{(2n)}_\e(y,q).
\end{equation}
Indeed, we have
$$
-\int_{\R^{2n}}q\cdot\nabla_x\tau_{(y,q)}W_\e\psi^\e_t
G^{(2n)}_\e(y,q)\,dydq=-\sqrt{\e}\nabla_x\cdot [W_\e\psi^\e_t\ast
\bar{G}_\e^{(2n)}].
$$
Although we will not use it here, let us mention that it is possible to derive a closed equation
(i.e. not involving $W_\e\psi^\e_t$) for $\tilde W_\e\psi^\e_t$ (see \cite{athanassoulis1} and
\cite{athanassoulis3}, \cite{athanassoulis2} for applications to the semiclassical limit
in strong topology).

\subsection{Assumptions on $U$}
\label{assumptionsonU}

We assume that $n=3M$, $x=(x_1,\ldots,x_M)\in (\R^3)^M$ and
$U=U_{s}+U_{b}$, with $U_{s}$ the (repulsive) Coulomb potential
\begin{equation}\label{Coulomb}
U_{s}(x)=\sum_{1\leq i<j\leq M}\frac{Z_iZ_j}{|x_i-x_j|},
\end{equation}
with $Z_i>0$, and $U_{b}$ globally bounded and
Lipschitz, with $\nabla U_{b}\in BV_{\rm loc}(\R^n;\R^n)$.

In this context the singular set $S$ of Section~\ref{sstable} and
Section~\ref{sgoodliou} is given by
$$
S=\bigcup_{1\leq i<j\leq M}S_{ij}\qquad\text{with}\qquad
S_{ij}:=\{x\in\R^n:\ x_i=x_j\}
$$
and therefore
\begin{equation}\label{lowerU}
U_{s}(x)\geq \frac{c}{{\rm dist\,}(x,S)}
\end{equation}
with $c>0$ depending only on the numbers $Z_i$ in \eqref{Coulomb}.

The vector field $\bb=(p,-\nabla U)$ satisfies the assumptions
(a)-(b) of Section~\ref{sstable} and the assumptions (i)-(iii) of
Section~\ref{sgoodliou}, so that the $\nu$-RLF in $\R^{2n}$ and the
$\nnu$-RLF in $\Probabilities{\R^{2n}}$ relative to $\bb$ exists and
are unique, and the stability result of Section~\ref{sstable} can be
applied, as we will show in Section~\ref{smain}.

\subsection{Estimates on solutions to \eqref{eq:Sch}}

\smallskip\noindent {\bf Conserved quantities.}
\begin{equation}\label{amfrja1}
\int_{\R^n}\frac{1}{2}|\e\nabla\psi^\e_t|^2+U|\psi^\e_t|^2\,dx=
\int_{\R^n}\frac{1}{2}|\e\nabla\psi^\e_0|^2+U|\psi^\e_0|^2\,dx
\qquad\forall\, t\in\R,
\end{equation}
\begin{equation}\label{amfrja2}
\int_{\R^n}|H_\e\psi^\e_t|^2\,dx=
\int_{\R^n}|H_\e\psi^\e_0|^2\,dx
\qquad\forall\, t\in\R.
\end{equation}

\smallskip\noindent {\bf A priori estimate.}
\cite[Lemma~5.1]{amfrja}.
\begin{equation}\label{amfrja3}
\sup_{t\in\R}\int_{\R^n} U_{s}^2|\psi^\e_t|^2\,dx\leq
\int_{\R^n}|H_\e\psi^\e_0|^2\,dx
+2\sup|U_{b}|\bigl(\int\langle\psi^\e_0,H_\e\psi^\e_0\rangle\,dx+\sup|U_{b}|\bigr).
\end{equation}

\smallskip\noindent
{\bf Tightness in space.} \cite[Lemma~3.3]{amfrja}.
\begin{equation}\label{alltight}
\sup_{t\in [-T,T]}\int_{\R^n\setminus B_{2R}}|\psi^\e_t|^2\,dx\leq
\int_{\R^n\setminus
B_R}|\psi^\e_0|^2(x)\,dx+cT\frac{1+\int\langle\psi^\e_0,H_\e\psi^\e_0\rangle\,dx}{R}
\end{equation}
with $c$ depending only on $n$.

\subsection{Estimates and convergence of $\Errordelta{U_b}{\psi}$}

In this section we prove estimates and convergence of the term
$\Errordelta{U_b}{\psi}$, as defined in \eqref{error_delta}. In
particular we use averaging with respect to the ``random'' parameter
$w$ to derive new estimates on $\Errordelta{V}{\psi^\e_w}$, with $V$
Lipschitz only, so that the estimates are applicable to $V=U_b$.

The first basic estimate on $\Errordelta{V}{\psi}$, for $\psi$ with
unit $L^2$ norm, can be obtained, when $V$ is Lipschitz, by
estimating the difference quotient in the square brackets in
\eqref{error_delta} with the Lipschitz constant:
\begin{equation}\label{primaerror}
\biggl|\int_{\R^{2n}}\Errordelta{V}{\psi}\phi\,dx\,dp\biggr|\leq\frac{1}{(2\pi)^n}\|\nabla
V\|_\infty\int_{\R^n}|y|\sup_{x\in\R^n}|{\cal F}_p\phi|(x,y)\,dy.
\end{equation}

In order to derive a more refined estimate we consider families
$\psi^\e_w$ indexed by a parameter $w\in W$, with $(W,{\mathcal
F},\P)$ probability space, satisfying:
\begin{equation}\label{meanbetter2}
\sup_{\e>0}\sup_{(x,p)\in\R^{2n}} \int_W\tilde
W_\e\psi^\e_w(x,p)\,d\P(w)<\infty,
\end{equation}
\begin{equation}\label{meanbetter1}
\sup_{\e>0}\sup_{x\in\R^n} \int_W|\psi^\e_w\ast
G^{(n)}_{\lambda\e^2}|^2(x)\,d\P(w)\leq
C(\lambda)<\infty\qquad\forall\lambda>0.
\end{equation}

Under these assumptions, our first convergence result reads as
follows:

\begin{theorem}[Convergence of error term, I]\label{extracanc}
Let $\psi_w^\e\in L^2(\R^n;\C)$ be normalized wavefunctions
satisfying \eqref{meanbetter2}, \eqref{meanbetter1} and let
$V:\R^n\to\R$ be Lipschitz. Then
\begin{equation}\label{erro_co}
\lim_{\e\to 0}\int_W
\biggl|\int_{\R^{2n}}\Errordelta{V}{\psi^\e_w}\phi\,dxdp+\int_{\R^{2n}}
\langle\nabla V,\nabla_p\phi\rangle\tilde
W_\e\psi^\e_wdxdp\biggr|d\P(w)=0\qquad\forall\phi\in
C^\infty_c(\R^{2n}).
\end{equation}
\end{theorem}
\begin{proof} The proof is achieved by a density argument. The first
remark is that linear combinations of tensor functions
$\phi(x,p)=\phi_1(x)\phi_2(p)$, with $\phi_i\in C^\infty_c(\R^n)$,
are dense for the norm considered in \eqref{primaerror}. In this
way, we are led to prove convergence in the case when
$\phi(x,p)=\phi_1(x)\phi_2(p)$. The second remark is that
convergence surely holds if $V$ is of class $C^2$ (by the arguments
in \cite{lionspaul}, \cite{amfrja}, see also the splitting argument
in the $y$ space in the proof of Theorem~\ref{extracoulomb}). Hence,
combining the two remarks and using the linearity of the error term
with respect to the potential $V$, we can prove convergence by a
density argument, by approximating $V$ uniformly and in $W^{1,2}$
topology on the support of $\phi_1$ by potentials $V_k\in C^2(\R^n)$
with uniformly Lipschitz constants; then, setting
$A_k=(V-V_k)\phi_1$ and choosing a sequence $\lambda_k$ in
Proposition~\ref{apriorinice} converging slowly to $0$, in such a
way that $\|\nabla A_k\|_2\to 0$ much faster than $1/C(\lambda_k)$,
we obtain
$$
\lim_{k\to\infty} \sup_{\e> 0}\int_W
\biggl|\int_{\R^{2n}}\Errordelta{V-V_k}{\psi^\e_w}(x,p)\phi_1(x)\phi_2(p)\,dxdp
\biggr\vert\,d\P(w)=0.
$$
As for the term in \eqref{erro_co} involving the Wigner transforms,
we can use \eqref{meanbetter2} to obtain that
$$
\limsup_{k\to\infty}\sup_{\e>0}\int_W \biggl\vert \int_{\R^{2n}}
\tilde W_\e\psi^\e_w\langle\nabla
(V-V_k),\nabla\phi_2\rangle\phi_1\,dxdp \biggr\vert\,d\P(w)
$$
can be estimated from above with a constant multiple of
$$
\limsup_{k\to\infty}\int_{\R^n}|\phi_1||\nabla V-\nabla V_k|\,dx
\int_{\R^n}|\nabla\phi_2|(p)\,dp=0.
$$
\end{proof}

We shall actually use the conclusion of Theorem~\ref{extracanc} in
the form
\begin{equation}\label{extracancbis}
\lim_{\e\to 0}\int_W
\biggl|\int_{\R^{2n}}\Errordelta{V}{\psi^\e_w}\phi\ast
G^{(2n)}_\e\,dxdp+\int_{\R^{2n}} \langle\nabla
V,\nabla_p\phi\rangle\tilde
W_\e\psi^\e_wdxdp\biggr|d\P(w)=0\qquad\forall\phi\in
C^\infty_c(\R^{2n})
\end{equation}
with $\phi$ replaced by $\phi\ast G^{(2n)}_\e$ in the first summand,
in the factor of $\Errordelta{V}{\psi^\e_w}$; this formulation is
equivalent thanks to \eqref{primaerror}.

\begin{proposition}[A priori estimate]\label{apriorinice}
Let $\psi_w^\e\in L^2(\R^n;\C)$ be unitary wavefunctions satisfying
\eqref{meanbetter1} and let $\phi_1,\,\phi_2\in C^\infty_c(\R^n)$.
Then, for all $V:\R^n\to\R$ Lipschitz and all $\lambda>0$, we have
that
\begin{equation}\label{apriori}
\int_W\biggl|\int_{\R^{2n}}
\Errordelta{V}{\psi_w^\e}(x,p)\phi_1(x)\phi_2(p)\,dxdp
\biggr|\,d\P(w)
\end{equation}
can be estimated from above with
\begin{eqnarray}\label{apriori1}
&&\|\phi_1\|_\infty\|\nabla V\|_\infty\int_{\R^n}|y||{\cal F}_p
\phi_2(y)-{\cal F}_p\phi_2\ast
G^{(n)}_\lambda(y)|\,dy+\sqrt{\lambda}\|\nabla A\|_\infty\|{\cal
F}_p\phi_2\|_1\int_{\R^n}|u|G^{(n)}_1(u)\,du\nonumber \\&&+
\sqrt{C(\lambda)}\|\nabla A\|_2\int_{\R^n}|z||{\cal
F}_p\phi_2|(z)\,dz+
\|V\|_\infty\|\nabla\phi_1\|_\infty\int_{\R^n}|y||{\cal
F}_p\phi_2\ast G^{(n)}_\lambda|(y)\,dy.
\end{eqnarray}
where $A:=V\phi_1$ and $C(\lambda)$ is given in \eqref{meanbetter1}.
\end{proposition}
\begin{proof} Set $\hat{\phi}_2={\cal F}_p\phi_2$; since \eqref{primaerror}
gives that
$$
\biggl|\int_{\R^{2n}}\Errordelta{V}{\psi^\e_w}\phi_1(x)\phi_2(p)\,dxdp-
\int_{\R^{2n}}\Errordelta{V}{\psi^\e_w}\phi_1(x)\phi_2(p)e^{-|p|^2\lambda/4}\,dxdp\biggr|$$
can be estimated from above with $\|\phi_1\|_\infty\|\nabla
V\|_\infty\int|y||\hat{\phi}_2(y)-\hat{\phi}_2\ast
G^{(n)}_\lambda(y)|\,dy$ we recognize the first error term in
\eqref{apriori1} and we will estimate the integral of
$\Errordelta{V}{\psi^\e_w}$ against
$\phi_1(x)\phi_2(p)e^{-|p|^2\lambda/4}$, namely
$$
\int_W\int_{\R^{2n}}\frac{V(x+\frac{\e}{2} y)-V(x-\frac{\e}{2}
y)}{\e}\phi_1(x)\hat{\phi}_2\ast G_\lambda^{(n)}(y)
\psi^\e_w(x+\frac{\e}{2} y)\overline{\psi^\e_w(x-\frac{\e}{2}
y)}dxdyd\P(w).
$$
In addition, we split this expression as the sum of three terms,
namely
\begin{equation}\label{defI}
I:= \int_W\int_{\R^{2n}}\frac{A(x+\frac{\e}{2} y) -A(x-\frac{\e}{2}
y)}{\e}\hat{\phi}_2\ast G^{(n)}_\lambda(y) \psi^\e_w(x+\frac{\e}{2}
y)\overline{\psi^\e_w(x-\frac{\e}{2} y)}\,dxdyd\P(w),
\end{equation}
\begin{equation}\label{defII}
II:=\int_W\int_{\R^{2n}}V(x+\frac{\e}{2}y)\frac{\phi_1(x)
-\phi_1(x+\frac{\e}{2} y)}{\e}\hat{\phi}_2\ast G^{(n)}_\lambda(y)
\psi^\e_w(x+\frac{\e}{2} y)\overline{\psi^\e_w(x-\frac{\e}{2}
y)}\,dxdyd\P(w),
\end{equation}
\begin{equation}\label{defIII}
III:=-\int_W\int_{\R^{2n}}V(x-\frac{\e}{2}y)\frac{\phi_1(x)
-\phi_1(x-\frac{\e}{2} y)}{\e}\hat{\phi}_2\ast G^{(n)}_\lambda(y)
\psi^\e_w(x+\frac{\e}{2} y)\overline{\psi^\e_w(x-\frac{\e}{2}
y)}\,dxdyd\P(w).
\end{equation}
The most difficult term to estimate is \eqref{defI}, since both
\eqref{defII} and \eqref{defIII} can be easily estimated from above
with
$\frac{1}{2}\|V\|_\infty\|\nabla\phi_1\|_\infty\int_{\R^n}|y||\hat{\phi}_2\ast
G^{(n)}_\lambda|(y)\,dy$. We first perform some manipulations of
this expression, omitting for simplicity the integration w.r.t. $w$;
then we will estimate the resulting terms taking \eqref{meanbetter1}
into account.

We expand the convolution product and make the change of variables
\eqref{change} to get
\begin{equation}\label{perm0}
\frac{1}{(\pi\lambda)^{n/2}\e^n}\int_{\R^n}\int_{\R^{2n}}
\frac{A(u)-A(u')}{\varepsilon} e^{-\scriptstyle{\frac{|\e
z-(u-u')|^2}{\e^2\lambda}}}
\psi^\e_w(u)\overline{\psi^\e_w(u')}\hat{\phi}_2(z)dudu'dz.
\end{equation}
Now, the term containing $A(u)$ is equal to
\begin{equation}\label{perm1}
\frac{1}{\e}\int_{\R^{2n}}(A\psi^\e_w)\ast
G^{(n)}_{\lambda\e^2}(u'+\e z)
\overline{\psi^\e_w(u')}\hat{\phi}_2(z)du'dz
\end{equation}
and the term containing $A(u')$ is equal to
\begin{equation}\label{perm2}
\frac{1}{\e}\int_{\R^{2n}}A(u')\psi^\e_w\ast
G^{(n)}_{\lambda\e^2}(u'+\e z)
\overline{\psi_w^\e(u')}\hat{\phi}_2(z)du'dz.
\end{equation}
Now, subtract \eqref{perm2} from \eqref{perm1} to get that
\eqref{perm0} equals $R_{w,1}^\e+R_{w,2}^{\e}$, where
$$
R_{w,1}^\e:=\frac{1}{\e}\int_{\R^{2n}}\bigl[(A\psi^\e_w)\ast
G^{(n)}_{\lambda\e^2}(u'+\e z)-A(u'+\e z)\psi^\e_w\ast
G^{(n)}_{\lambda\e^2}(u'+\e z)\bigr]
\overline{\psi}^\e_w(u')\hat{\phi}_2(z)du'dz
$$
and
$$
R_{w,2}^\e:=\frac{1}{\e}\int_{\R^{2n}}[A(u'+\e
z)-A(u')]\psi^\e_w\ast G^{(n)}_{\lambda\e^2}(u'+\e z)
\overline{\psi_w^\e(u')}\hat{\phi}_2(z)du'dz.
$$
Thus, the apriori estimate on the expression in \eqref{defI} can be
achieved by estimating the integrals of the error terms $R^\e_{w,i}$
w.r.t. $w$.

Writing $R_{w,1}^\e$ in the form
$$
\int_{\R^n}\hat\phi_2(z)\int_{\R^n\times\R^n} \frac{A(u'+\e
z-u)-A(u'+\e z)}{\e}
G^{(n)}_{\lambda\e^2}(u)\overline{\psi^\e_w(u')}\psi^\e_w(u'+\e
z-u)\,dudu'dz
$$
we can estimate from above $\int_W |R_{w,1}^\e|\,d\P(w)$ by
$$
\|\nabla
A\|_\infty\int_W\int_{\R^n}|\hat\phi_2|(z)\int_{\R^n\times\R^n}
\frac{|u|}{\e}G^{(n)}_{\lambda\e^2}(u)|\psi^\e_w|(u')|\psi^\e_w|(u'+\e
z-u)\,dudu'dzd\P(w)
$$
and then by
$$
\sqrt{\lambda}\|\nabla
A\|_\infty\int_W\int_{\R^n}|\hat{\phi}_2|(z)\int_{\R^n\times\R^n}
\eta_\e(u)|\psi^\e_w|(u')|\psi^\e_w|(u'+\e z-u)\,dudu'dzd\P(w)
$$
where $\eta_\e(u):=G^{(n)}_{\lambda\e^2}(u)|u|/(\sqrt{\lambda}\e)$
is a family of convolution kernels uniformly bounded in $L^1$ by
$\int|u|G^{(n)}_1(u)\,du$. Using the convolution estimate
$\|a\ast\eta_\e\|_2\leq\|a\|_2\|\eta_\e\|_1$ we can finally bound
this term with $\sqrt{\lambda}\|\nabla
A\|_\infty\|\hat\phi_2\|_1\int|u|G^{(n)}_1(u)\,du$.

We can estimate from above $\int_W |R_{w,2}^\e|\,d\P(w)$ using
\eqref{meanbetter1} to get
$$
\sqrt{C(\lambda)}\int_{\R^n}|\hat{\phi}_2|(z)\int_{\R^n}\frac{|A(u'+\e
z)-A(u')|}{\e}\sqrt{\int_W|\psi^\e_w|^2(u')\,d\P(w)}\,du'dz.
$$
Then we can use the standard $L^2$ estimate on difference quotients
of $W^{1,2}$ functions to bound this last expression with
$$
\sqrt{C(\lambda)}\|\nabla A\|_2\int_{\R^n}|z||\hat{\phi}_2|(z)\,dz.
$$
This completes the estimate of the term in \eqref{defI} and the
proof.
\end{proof}

\subsection{Estimates and convergence of $\Errordelta{U_s}{\psi}$}

In the case of the Coulomb potential we follow a specific argument
borrowed from \cite[proof of Theorem~1.1(ii)]{amfrja}), based on the
inequality
\begin{equation}\label{euclid}
\biggl|\frac{1}{|z+w/2|}-\frac{1}{|z-w/2|}\biggr| \leq
\frac{|w|}{|z+w/2||z-w/2|}
\end{equation}
with $z=(x_i-x_j)\in\R^3$, $w=\e(y_i-y_j)\in\R^3$. By estimating the
difference quotients of $U_s$ as in \eqref{euclid} we obtain:
\begin{equation}\label{secondaerror}
\biggl|\int_{\R^d}\Errordelta{U_{s}}{\psi}\phi\,dx\,dp\biggr|\leq
C_*\int_{\R^n}|y|\sup_{x'}|{\cal
F}_p\phi|(x',y)\,dy\int_{\R^n}U_{s}^2|\psi|^2\,dx,
\end{equation}
with $C_*$ depending only on the numbers $Z_i$ in \eqref{Coulomb}.

Now we can state the convergence of $\Errordelta{U_s}{\psi^\e}$; the
particular form of the statement, with convolution on $\phi$ on one
side and convolution on $W_\e\psi^\e$ on the other side (namely the
Husimi transform), is motivated by the goal we have in mind, namely
the fact that the Husimi transforms asymptotically satisfy the
Liouville equation.

\begin{theorem}[Convergence of error term, II]\label{extracoulomb}
Let $\psi^\e\in L^2(\R^n;\C)$ be unitary wavefunctions satisfying
\begin{equation}\label{meanbetter3}
\sup_{\e>0}\int_{\R^n}U_s^2|\psi^\e|^2\,dx<\infty.
\end{equation}
Then
\begin{equation}\label{erro_co2}
\lim_{\e\to 0}\int_{\R^{2n}}\Errordelta{U_s}{\psi^\e}\phi\ast
G^{(n)}_\e\,dxdp+\int_{\R^{2n}} \langle\nabla
U_s,\nabla_p\phi\rangle\tilde W_\e\psi^\e
\,dxdp=0\quad\forall\phi\in C^\infty_c(\R^{2n}\setminus
(S\times\R^n)\bigr).
\end{equation}
\end{theorem}
\begin{proof} First of all, we see that we can apply
\eqref{uniformly} with $\varphi=\langle\nabla
U_s,\nabla_p\phi\rangle$ to replace the integrals $\int_{\R^{2n}}
\langle\nabla U_s,\nabla_p\phi\rangle\tilde W_\e\psi^\e\,dxdp$ with
$\int_{\R^{2n}} \langle\nabla U_s,\nabla_p\phi\rangle
W_\e\psi^\e\,dxdp$ in the verification of \eqref{erro_co2}.
Analogously, using \eqref{meanbetter3} and \eqref{secondaerror} we
see that we can replace
$\int_{\R^{2n}}\Errordelta{U_s}{\psi^\e}\phi\ast G^{(n)}_\e\,dxdp$
with $\int_{\R^{2n}}\Errordelta{U_s}{\psi^\e}\phi\,dxdp$. Thus, we
are led to show the convergence
\begin{equation}\label{erro_co3}
\lim_{\e\to
0}\int_{\R^{2n}}\Errordelta{U_s}{\psi^\e}\phi\,dxdp+\int_{\R^{2n}}
\langle\nabla U_s,\nabla_p\phi\rangle W_\e\psi^\e
\,dxdp=0\quad\forall\phi\in C^\infty_c(\R^{2n}\setminus
(S\times\R^n)\bigr).
\end{equation}
Since
$$
\int_{\R^{2n}}\Errordelta{U_s}{\psi^\e}\phi\,dxdp=
-\frac{i}{(2\pi)^n}\int_{\R^{2n}}
\frac{U_s(x+\frac{\e}{2}y)-U_s(x-\frac{\e}{2}y)}{\e}
\psi^\e(x+\frac{\e}{2}y)\overline{\psi(x-\frac{\e}{2}y)}{\cal
F}_p\phi(x,y)\,dxdy
$$
we can split the region of integration in two parts, where
$\sqrt{\e}|y|>1$ and where $\sqrt{\e}|y|\leq 1$. The contribution of
the first region can be estimated as in \eqref{secondaerror}, with
$$
C_*\int_{\{\sqrt{\e}|y|>1\}}|y|\sup_{x'}|{\cal
F}_p\phi|(x',y)\,dy\int_{\R^n}U_{s}^2|\psi^\e|^2\,dx,
$$
which is infinitesimal, using \eqref{meanbetter3} again, as $\e\to
0$. Since
$$
\frac{U_s(x+\frac{\e}{2}y)-U_s(x-\frac{\e}{2}y)}{\e}\to\langle\nabla
U_s(x),y\rangle $$ uniformly as $\sqrt{\e}|y|\leq 1$ and $x$ belongs
to a compact subset of $\R^n\setminus S$, the contribution of the
second part is the same as that of
$$
-\frac{i}{(2\pi)^n} \int_{\R^{2n}} \langle\nabla U_s(x),y\rangle
\psi^\e(x+\frac{\e}{2}y)\overline{\psi(x-\frac{\e}{2}y)}{\cal
F}_p\phi(x,y)\,dxdy
$$
which coincides with
$$
-\int_{\R^{2n}}\langle\nabla U_s,\nabla_p\phi\rangle
W_\e\psi^\e(x,p)\,dxdp.
$$
\end{proof}

\section{$L^\infty$-estimates on averages of $\psi$}\label{sHusimi}

In this section we consider a family of solutions $\psi^\e_{t,w}$ to
the Schr\"odinger equation \eqref{eq:Sch} indexed by a parameter
$w$, and derive new estimates on their averages. In particular we
obtain pointwise upper bounds on Husimi transforms.

One of the main advantages of the Husimi transform is that it is
non-negative: indeed, with the change of variables \eqref{change}
and simple computations (see \cite{lionspaul} for more details), it
can be written as
\begin{equation}\label{rapptildeW}
\tilde
W_\e\psi(y,p)=\frac1{(2\pi)^n}\<\rho^\psi\phi^\e_{y,p},\phi^\e_{y,p}\>
=\frac1{(2\pi)^n}|\<\psi,\phi^\e_{y,p}\>|^2,
\end{equation}
where $\<\cdot,\cdot\>$ is the scalar product on $L^2(\R^n;\C)$,
\begin{equation}\label{phiep}
\phi^\e_{y,p}(x):=\frac{1}{\e^{n/2}}\frac{1}{(\pi\e)^{n/4}}
e^{-|x-y|^2/(2\e)}e^{i(p\cdot x)/\e}\in L^2(\R^n;\C),
\end{equation}
and $\rho^\psi:L^2(\R^n;\C)\to L^2(\R^n;\C)$ is the orthogonal
projector onto $\psi\in L^2(\R^n;\C)$:
$$
[\rho^\psi \phi](x):= \left(\int_{\R^n} \phi(x') \ov{\psi(x')}
 \,dx'\right)\psi(x).
$$

\begin{proposition}[$L^\infty$ estimates]\label{ppaul}
Let $\psi^\e_w\in L^2(\R^n;\C)$ be satisfying the operator
inequalities
$$
\frac{1}{\e^n}\int_W \rho^{\psi^\e_w}\,d\P(w)\leq C{\rm
Id}\qquad\forall\e>0.
$$
Then:
\begin{itemize}
\item[(a)] for all $y\in\R^n$ and $\e,\,\lambda>0$ we have
$$
\int_W|\psi_w^\e\ast G^{(n)}_{2\lambda\e^2}|^2(y)\,d\P(w)\leq
\frac{C}{\lambda^{n/2}};
$$
\item[(b)] for all $(y,p)\in\R^{2n}$ and $\e>0$ we have
$$
\int_W\tilde W_\e\psi_w^\e(y,p)\,d\P(w)\leq C.
$$
\end{itemize}
\end{proposition}
\begin{proof} The proof of (a) follows by applying the uniform operator
inequality to the functions $(2\e)^{n/2}(\pi\lambda)^{n/4}
G^{(n)}_{2\lambda\e^2}(\cdot-y)$, whose $L^2$ norm is 1, to get
$$
\e^n\lambda^{n/2}\int_W |\psi^\e_w\ast
G^{(n)}_{2\lambda\e^2}|^2(y)\,d\P(w)\leq C\e^n.
$$
The proof of (b) is analogous, it is based on \eqref{rapptildeW} and
on the insertion of the functions $\phi^\e_{y,p}$ in \eqref{phiep}
in the operator inequality, taking into account that
$\|\phi^\e_{y,p}\|_2=\e^{-n/2}$.
\end{proof}

The assumption made in Proposition~\ref{ppaul} is compatible with
the families of wavefunctions given in \eqref{alphaless1}, i.e.
\begin{equation}\label{alpha1}
\psi^\e_w(x)=\e^{-n\alpha/2}\phi_0\Bigl(\frac{x-x_0}{\e^\alpha}\Bigr)e^{i(x\cdot
p_0)/\e}\qquad \phi_0\in C^2_c(\R^n),\qquad 0<\alpha<1
\end{equation}
with $w=(x_0,p_0)$. Indeed, in this case one can choose $W=\R^{2n}$
with the Borel $\sigma$-algebra and $\P=\rho\Leb{2n}$, with $\rho\in
L^1\cap L^\infty$, see \cite{figpaul} for details. In the extreme
case $\alpha=1$ no average w.r.t. $p_0$ is needed and one can fix it
and choose $W=\R^n$, obtaining convergence for almost all $x_0$, so
to speak. The other extreme case $\alpha=0$, corresponding to
concentration in momentum, is analogous.

\section{Main convergence result}\label{smain}

In this section we combine the theory developed in
Sections~\ref{notation}--\ref{sgoodliou} with the estimates of the
Section~\ref{estimates} and Section~\ref{sHusimi}, to obtain
convergence of the Wigner/Husimi transforms of solutions to
\eqref{eq:Sch}. In particular we shall apply Theorem~\ref{tstable}.

We consider the assumptions on $U$ stated in
Section~\ref{assumptionsonU} and ``random'' initial data
$\psi^\e_{0,w}\in H^2(\R^n;\C)$ with unit $L^2$ norm in
\eqref{eq:Sch} indexed by $w\in W$, where $(W,{\cal F},\P)$ is a
suitable probability space. Denoting by $\psi^\e_{t,w}$ the
corresponding Schr\"odinger evolutions, the basic assumptions we
need for the initial data are
\begin{equation}\label{assum1}
\sup_{\e>0}\int_W\int_{\R^n}|H_\e\psi^\e_{0,w}|^2\,dx\,d\P(w)<\infty,\qquad
\lim_{R\uparrow\infty}\sup_{\e>0}\int_W\int_{\R^n\setminus
B_R}|\psi^\e_{0,w}|^2\,dx\,d\P(w)=0;
\end{equation}
\begin{equation}\label{assum2}
\frac{1}{\e^n}\int_W\rho^{\psi^\e_{0,w}}\,d\P(w)\leq C{\rm Id}\quad
\text{with $C$ independent of $\e$};
\end{equation}
\begin{equation}\label{assum3}
i(w):=\lim_{\e\downarrow 0}\tilde W_\e\psi^\e_{0,w}\Leb{d}
\quad\text{exists in $\Probabilities{\R^d}$ for $\P$-a.e. $w\in W$.}
\end{equation}
As we discussed in Section~\ref{sHusimi}, \eqref{assum1},
\eqref{assum2}, \eqref{assum3} are compatible with several natural
families of initial conditions, for instance those described at the
end of the introduction, see \eqref{alphaless1} or \eqref{alpha1}.
In addition, the unitary character of the Schr\"odinger evolution
immediately gives
\begin{equation}\label{assum2bis}
\frac{1}{\e^n}\int_W\rho^{\psi^\e_{t,w}}\,d\P(w)\leq C{\rm Id}
\qquad\forall \e>0,\,\,t\geq 0,
\end{equation}
where $C$ is the same constant as in \eqref{assum2}.

We state \eqref{final} below using the Husimi transforms, but
\eqref{uniformly} can be used to show convergence of Wigner
transforms, in the form used in \eqref{limeps} in the introduction.

\begin{theorem}\label{tmain2} For $U$ as in Section \ref{assumptionsonU}, and
under assumptions \eqref{assum1}, \eqref{assum2}, \eqref{assum3}, we
have
\begin{equation}\label{final}
\lim_{\e\to 0} \int_W\sup_{t\in [-T,T]}d_{{\mathscr
P}}\bigl(\tilde{W}_\e\psi^\e_{t,w},\mmu(t,i(w))\bigr) \,d\P(w)=0,
\end{equation}
for all $T>0$, where
$\nnu=i_\sharp\P\in\Probabilities{\Probabilities{\R^{2n}}}$ and
$\mmu(t,\mu)$ is the $\nnu$-RLF in \eqref{svitla}.
\end{theorem}
\begin{proof} Our goal is to apply Theorem~\ref{tstable}
(with a continuous parameter $\e$) and Remark~\ref{rven} with
$i_\e(w):=\tilde W_\e\psi^\e_{0,w}\Leb{2n}$ and
$\mmu_\e(t,i_\e(w))=\tilde W_\e\psi^\e_{t,w}\Leb{2n}$. The
convergence \eqref{final} will be a direct consequence of
\eqref{cetraro1}. We shall work in the time interval $[0,T]$, the
proof in the time interval $[-T,0]$ being the same, up to a time
reversal. First of all we notice that \eqref{assum1} and
\eqref{amfrja2} give
\begin{equation}\label{assum4}
\sup_{\e>0}\sup_{t\in\R}\int_W\int_{\R^n}|H_\e\psi^\e_{t,w}|^2\,dx\,d\P(w)<\infty.
\end{equation}
In particular, by an integration by parts, we have also
\begin{equation}\label{assum5}
\sup_{\e>0}\sup_{t\in\R}\int_W\int_{\R^n}|\e\nabla\psi^\e_{t,w}|^2\,dx\,d\P(w)<\infty.
\end{equation}

\noindent (i) (asymptotic regularity). By \eqref{assum2bis} and
Proposition~\ref{ppaul}(b) we have the uniform estimate (in $\e$,
$t$ and $(x,p)$)
\begin{equation}\label{unreg}
\int_W\tilde W_\e\psi^\e_{t,w}(x,p)\,d\P(w)\leq C.
\end{equation}
In particular we have uniform and not only asymptotic regularity,
therefore Remark~\ref{rven} applies.

\item[(ii)] (uniform decay away from the singularity). We check
\eqref{ali4} with $\beta=2$ and $S$ equal to the singular set of
$U_{s}$, namely
\begin{equation}\label{ali44}
\sup_{\delta>0}\limsup_{\e\to
0}\int_W\int_0^T\int_{B_R}\frac{1}{{\rm dist}^2(x,S)+\delta}\tilde
W_\e\psi^\e_{t,w}\,dx\,dp\,dt \,d\P(w)<\infty.
\end{equation}
We use \eqref{compamarginals1} and the inequality
$$
\frac{1}{{\rm dist}^2(x,S)+\delta}\ast G^{(n)}_\e\leq \frac{1}{{\rm
dist}^2(x,S)},
$$
which holds in $B_R$ for $\e<\e(\delta,R)$ to deduce \eqref{ali44}
from
\begin{equation}\label{ali4444}
\limsup_{\e\to 0}\int_W\int_0^T\int_{\R^n}\frac{1}{{\rm
dist}^2(x,S)}|\psi^\e_{t,w}|^2\,dx\,dt \,d\P(w)<\infty.
\end{equation}
In turn, this inequality follows by \eqref{amfrja3} and
\eqref{lowerU}, taking \eqref{assum1} into account.

\noindent (iii) (space tightness). We have to check that for all
$\delta>0$ it holds:
$$
\lim_{R\to\infty} \P\biggl(\Bigl\{w\in W:\ \sup_{\e>0}\sup_{t\in
[0,T]}\int_{\R^{2n}\setminus B_R} \tilde{W}_\e\psi^\e_{t,w}\,dx\,dp
>\delta\Bigr\}\biggr)=0.
$$
Considering the cube $C_R$ containing $B_R$, this tightness property
can be checked separately for the first and the second marginals of
$\tilde{W}_\e\psi^\e_{t,w}$; using \eqref{compamarginals1},
\eqref{compamarginals2}, it is not hard to see that it suffices to
check the analogous property for the marginals of the corresponding
Wigner transforms; for the first marginals, tightness is a direct
consequence of \eqref{alltight} and \eqref{assum1}. For the second
marginals, we use \eqref{assum5} and the identity
\begin{equation}\label{tightpp}
\int_{\R^n\times\R^n}|p|^2 W_\e\psi\,dx\,dp=
\int_{\R^n}\biggl|\frac{1}{(2\pi\e)^{n/2}}\hat\psi(p/\e)\biggr|^2|p|^2\,dp=
\int_{\R^n}|\e\nabla\psi|^2\,dx
\end{equation}
with $\psi=\psi_{t,w}^\e$.

\noindent (iv) (time tightness). We need to show that for all
$\phi\in C_c^\infty(\R^{2n})$ it holds
$$
\lim_{M\uparrow\infty}\P \biggl(\Bigr\{w\in W:\
\int_0^T\biggl|\biggr(\int_{\R^{2n}}\phi \tilde
W_\e\psi^\e_{t,w}\,dx\,dp\biggr)'\biggr|\,dt>M\Bigr\}\biggr)=0;
$$
uniformly in $\e$. Equivalently, we can consider the limit
\begin{equation}\label{elena}
\lim_{M\uparrow\infty}\P\biggl(\Bigl\{w\in W:\
\int_0^T\biggl|\biggr(\int_{\R^{2n}}\phi_\e
W_\e\psi^\e_{t,w}\,dx\,dp\biggr)'\biggr|\,dt>M\Bigr\}\biggr)=0,
\end{equation}
where $\phi_\e=\phi\ast G_\e^{(2n)}$. According to
\eqref{Wigner_calc_in}, the time derivative in the formula above
consists of two terms, $\int\langle p,\nabla_x\phi_\e\rangle
W_\e\psi^\e_{t,w}\,dx\,dp$ and
$\int\Errordelta{U}{\psi^\e_{t,w}}\phi_\e\,dx\,dp$ and we need only
to show a property analogous to \eqref{elena} for these two terms.
Since $\phi \in C_c^\infty(\R^{2n})$, $\|\langle
p,\nabla_x\phi_\e\rangle\|_{{\cal A}}$ are easily seen to be
uniformly bounded, hence the first term can be estimated using
\eqref{lionspaul}. The second term can be estimated using
\eqref{primaerror} for $U_{b}$ and \eqref{secondaerror} for $U_{s}$,
taking \eqref{amfrja3} and \eqref{assum1} into account.

\noindent (v) (limit continuity equation). We have to show that
$$
\lim_{\e\downarrow 0}\int_W
\biggl|\int_0^T\biggl[\varphi'(t)\int_{\R^{2n}}\phi \tilde
W_\e\psi^\e_{t,w}\,dxdp+\varphi(t)\int_{\R^{2n}}
\langle\bb,\nabla\phi\rangle \tilde
W_\e\psi^\e_{t,w}\,dxdp\biggr]\,dt\biggr|\,d\P(w)=0
$$
for all $\phi\in C^\infty_c\bigl(\R^{2n}\setminus
(S\times\R^n)\bigr)$, $\varphi\in C^\infty_c(0,T)$. Taking
\eqref{pdehusimi} into account, this is implied by the validity of
the limits
\begin{equation}\label{limit1}
\lim_{\e\downarrow 0} \sup_{t\in [0,T]}\int_W \biggl|
\int_{\R^{2n}}\Errordelta{U}{\psi^\e_{t,w}}\phi\ast
G^{(2n)}_\e\,dxdp+\int_{\R^{2n}} \langle\nabla
U,\nabla_p\phi\rangle\tilde
W_\e\psi^\e_{t,w}\,dxdp\biggr|\,d\P(w)=0,
\end{equation}
\begin{equation}\label{limit2}
\lim_{\e\downarrow 0} \sqrt{\e}\int_W\int_0^T|\varphi(t)|
\biggl|\int_{\R^{2n}} \phi\nabla_x\cdot [W_\e\psi^\e_{t,w}\ast
\bar{G}_\e^{(2n)}]\,dxdp\biggr|\,dtd\P(w)=0.
\end{equation}

\noindent {\bf Verification of \eqref{limit1}.} We can consider
separately the contributions of $U_b$ and $U_s$. For the $U_b$
contribution we apply Theorem~\ref{extracanc}, in the form stated in
\eqref{extracancbis}; the assumptions \eqref{meanbetter2} and
\eqref{meanbetter1} of that theorem are fulfilled in view of
\eqref{assum2} and Proposition~\ref{ppaul}. For the $U_s$
contribution we apply \eqref{erro_co2} of
Theorem~\ref{extracoulomb}; the assumption \eqref{meanbetter3} of
that theorem is fulfilled in view of assumption \eqref{assum1} on
the initial data and \eqref{amfrja3}, ensuring propagation in time.

\noindent {\bf Verification of \eqref{limit2}.} This is easy, taking
into account the fact that $$\int_{\R^{2n}}\langle
W_\e\psi^\e_{t,w}\ast \bar{G}_\e^{(2n)},\nabla_x\phi\rangle\,dxdp=
-\int_{\R^{2n}} W_\e\psi^\e_{t,w} \nabla_x\cdot
[\phi\ast\bar{G}_\e^{(2n)}]\,dxdp$$ are uniformly bounded because
$\bar{G}_\e^{(2n)}$, defined in \eqref{pdehusimibis}, are uniformly
bounded in $L^1(\R^n)$.
\end{proof}


\begin{thebibliography}{99}

\bibitem{ambrosio} {\sc L.Ambrosio:} {\em Transport equation and Cauchy
problem for $BV$ vector fields.} Invent. Math., {\bf 158} (2004),
227--260.

\bibitem{cetraro} {\sc L.Ambrosio:} {\em Transport equation and Cauchy problem
for non-smooth vector fields.} Lecture Notes in Mathematics
``Calculus of Variations and Non-Linear Partial Differential
Equations'' (CIME Series, Cetraro, 2005) {\bf 1927}, B. Dacorogna,
P. Marcellini eds., 2--41, 2008.
%
%
%
\bibitem{bologna} {\sc L.Ambrosio, G.Crippa:} {\em
Existence, uniqueness, stability and differentiability properties of
the flow associated to weakly differentiable vector fields.} UMI
Lecture Notes, Springer, in press.


\bibitem{amgisa} {\sc L.Ambrosio, N.Gigli, G.Savar\'e:}
{\em Gradient flows in metric spaces and in the Wasserstein space of
probability measures.} Lectures in Mathematics, ETH Zurich,
Birkh\"auser, 2005.

\bibitem{amfiga} {\sc L.Ambrosio, A.Figalli}: {\em On flows associated to
Sobolev vector fields in Wiener spaces: an approach \`a la
DiPerna-Lions.} J. Funct. Anal., {\bf 256} (2009), 179-214.

\bibitem{amfiga2} {\sc L.Ambrosio, A.Figalli}: {\em
Almost everywhere well-posedness of continuity equations with
measure initial data.} Preprint, 2009, to appear in CRAS.

\bibitem{amfrja} {\sc L.Ambrosio, G.Friesecke, J.Giannoulis}:
{\em Passage from quantum to classical molecular dynamics in the
presence of Coulomb interactions.} Preprint, 2009, to appear in
Comm. PDE.

\bibitem{athanassoulis1}
{\sc A.Athanassoulis, N.Mauser, T.Paul:} {\em  Coarse-scale
representations and smoothed Wigner transforms.} Journal de
Math\'ematiques Pures et Appliqu\'ees, {\bf 91} (2009), 296-338.

\bibitem{athanassoulis3}
{\sc A.Athanassoulis, T.Paul:} {\em  Strong phase-space
semiclassical asymptotics.} arXiv:1002.1371v2.

\bibitem{athanassoulis2}
{\sc A.Athanassoulis, T.Paul:} {\em Regularization of certain
ill-posed semiclassical limits.} work in progress.


\bibitem{bogachevII} {\sc V.Bogachev:} {\em Measure Theory, Voll. I
and II.} Springer, 2007.

\bibitem{bouchut} {\sc F.Bouchut:}
{\em Renormalized solutions to the Vlasov equation with
coefficients of bounded variation.} Arch. Ration. Mech. Anal.,
{\bf 157} (2001), 75--90.

%

\bibitem{lions} {\sc R.J.DiPerna, P.L.Lions:}
{\em Ordinary differential equations, transport theory and Sobolev
spaces.} Invent. Math., {\bf 98} (1989), 511--547.

\bibitem{Fermanian1}
{\sc Fermanian-Kammerer, P.G\'erard:}
{\em Mesures semi-classiques et croisement de modes.}
Bull. Soc. Math. France, {\bf 130} (2002), 123--168.

\bibitem{Fermanian2}
{\sc Fermanian-Kammerer, P.G\'erard:}
{\em A Landau-Zener formula for non-degenerated involutive codimension 3 crossings.}
Ann. Henri Poincar\'e, {\bf 4} (2003), 513--552.

\bibitem{figalli} {\sc A.Figalli:}
{\em Existence and uniqueness of martingale solutions for SDEs with
rough or degenerate coefficients.} J. Funct. Anal., {\bf 254}
(2008), 109--153.

\bibitem{figpaul} {\sc A.Figalli, T.Paul:} work in progress.

\bibitem{Fr03} {\sc G.~Friesecke:}
{\em The Multiconfiguration Equations for Atoms and Molecules:
Charge quantization and existence of solutions.} Arch. Rat. Mech.
Analysis {\bf 169} (2003), 35--71.

\bibitem{gerard} {\sc P.G\'erard:} {\em Mesures semi-classiques et ondes
de Bloch.} Seminaire sur les \'Equations aux D\'eriv\'ees Partielles, 1990-1991. Exp. No.
XVI, 19 pp., \'Ecole Polytechnique, Palaiseau, 1991.

\bibitem{gerard3}
{\sc P.G\'erard, P.A.Markowich, N.J.Mauser, F.Poupaud:} {\em
Homogenization limits and Wigner transforms.} Comm. Pure Appl. Math.
{\bf 50} (1997), 323--379.

\bibitem{hauray1} {\sc M.Hauray:}
{\em On Liouville transport equation with potential in $BV_{\rm
loc}$.} Comm. Partial Differential Equations, {\bf 29} (2004),
207--217.

\bibitem{lionspaul} {\sc P.L.Lions, T.Paul:} {\em Sur les mesures de Wigner.}
Rev. Mat. Iberoamericana, {\bf 9} (1993), 553--618.

\bibitem{lions2} {\sc P.L.Lions:}
{\em Mathematical topics in fluid mechanics, Vol. I: incompressible
models.} Oxford Lecture Series in Mathematics and its applications,
{\bf 3} (1996), Oxford University Press.

\bibitem{lions3} {\sc P.L.Lions:}
{\em Mathematical topics in fluid mechanics, Vol. II: compressible
models.} Oxford Lecture Series in Mathematics and its applications,
{\bf 10} (1998), Oxford University Press.

\bibitem{lions1} {\sc P.L.Lions:}
{\em Sur les \'equations diff\'erentielles ordinaires et les
\'equations de transport.} C. R. Acad. Sci. Paris S\'er. I, {\bf
326} (1998), 833--838.

\bibitem{Martinez} {\sc A.Martinez:}
{\em An Introduction to Semiclassical and Microlocal Analysis.} Springer-Verlag, 2002

\bibitem{strvar} {\sc D.W.Stroock, S.R.S.Varadhan}
{\em Multidimensional diffusion processes.} Grundlehren der
Mathematischen Wissenschaften [Fundamental Principles of
Mathematical Sciences], 233. Springer-Verlag, Berlin-New York, 1979

\bibitem{Sturm1} {\sc K.T.Sturm, M.Von Renesse:} {\em Entropic measure and
Wasserstein diffusion.} Ann. Probability, to appear.

\bibitem{Sturm2} {\sc K.T.Sturm:} {\em Entropic measure on
multidimensional spaces.} arXiv: 0901.1815.



\bibitem{villani2}
{\sc C.Villani:} {\em Optimal transport: old and new.} Volume 338
of {\em Grundlehren der Mathematischen Wissenschaften [Fundamental
Principles of Mathematical Sciences]}. Springer, New York, 2009.


\end{thebibliography}
\end{document}